\documentclass[11pt, letterpaper,leqno]{amsart}

\usepackage{amsmath,amsthm,amsrefs,amssymb,hyperref}
\usepackage[usenames,dvipsnames]{xcolor}
\usepackage{todonotes}
\usepackage{fancyhdr}

\numberwithin{equation}{section}

\newtheorem{lemma}{Lemma}[section]
\newtheorem{theorem}{Theorem}[section]

\theoremstyle{definition}

\newtheorem{innercustomgeneric}{\customgenericname}
\providecommand{\customgenericname}{}
\newcommand{\newcustomtheorem}[2]{%
	\newenvironment{#1}[1]
	{%
		\renewcommand\customgenericname{#2}%
		\renewcommand\theinnercustomgeneric{##1}%
		\innercustomgeneric
	}
	{\endinnercustomgeneric}
}

\newcustomtheorem{customthm}{Theorem}
\newcustomtheorem{customcorollary}{Corollary}
\newcustomtheorem{customlemma}{Lemma}

\DeclareMathOperator{\D}{\mathbb{D}}
\DeclareMathOperator{\N}{\mathbb{N}}
\DeclareMathOperator{\C}{\mathbb{C}}
\DeclareMathOperator{\R}{\mathbb{R}}
\DeclareMathOperator{\Ha}{\mathbb{H}}
\DeclareMathOperator{\dist}{\mathrm{dist}}

\begin{document}
\title[Monotonicity of the speeds of semigroups]{On the monotonicity of the speeds for semigroups of holomorphic self-maps of the unit disk}

\author{Dimitrios Betsakos}  
\address{Department of Mathematics, Aristotle University of Thessaloniki, 54124, Thessaloniki, Greece}
\email{betsakos@math.auth.gr}  

\author{Nikolaos Karamanlis}
\address{Department of Mathematics and Applied Mathematics, University of Crete, Heraklion 70013, Crete, Greece}
\email{karamanlisn@gmail.com} 
\thanks{The second author is supported by the Hellenic Foundation for Research and Innovation, Project HFRI-FM17-1733.}
 %\thanks{}


%\date{}

%\footnote{The second author is supported by the Hellenic Foundation for Research and Innovation, Project HFRI-FM17-1733.}

\begin{abstract} 
We study semigroups $(\phi_t)_{t\geq 0}$ of holomorphic self-maps  of the unit disk with Denjoy-Wolff point on the boundary. We show that the orthogonal speed of such semigroups is a strictly increasing function. This answers a question raised by F. Bracci, D. Cordella, and M. Kourou, and implies a domain monotonicity property for orthogonal speeds conjectured by Bracci. We  give an example of a semigroup such that its total speed is not eventually increasing. We also provide another example of a semigroup having total speed of  a certain asymptotic behavior, thus answering another question of Bracci.
\end{abstract}

\maketitle

\section{Introduction}\label{intro}
A one parameter family $(\phi_t)_{t\geq 0}$ of holomorphic self-maps of the unit disk $\mathbb D$ is called a semigroup when \\
(a) $\phi_0$ is the identity;\\
(b) for all $t,s\geq 0$, $\phi_{t+s}=\phi_t\circ \phi_s$;\\
(c) for every $s\geq 0$ and every $z\in \mathbb D$, $\lim_{t\to s}\phi_t(z)=\phi_s(z)$.\\
For a comprehensive presentation of the rich theory of semigroups, we refer to the recent treatise \cite{BCD}. Here, we review only the basic properties we will need later.

\medskip

We restrict ourselves to {\it non-elliptic} semigroups, namely semigroups such that for every $t>0$, $\phi_t$ has no fixed points in $\mathbb D$. Then there exists a point $\tau\in \partial \mathbb D$, the {\it Denjoy-Wolff point} of the semigroup, such that for every $z\in\mathbb D$,
\begin{equation}\label{i1}
\lim_{t\to\infty}\phi_t(z)=\tau.
\end{equation} 
An important representation of a semigroup $(\phi_t)_{t\geq 0}$ is the following: There exists a unique conformal mapping $h:\mathbb D\to\mathbb C$ with $h(0)=0$ and 
\begin{equation}\label{i2}
\phi_t(z)=h^{-1}(h(z)+t),\;\;\;\;\;z\in \mathbb D,\;t\geq 0.
\end{equation} 
The function $h$ is the {\it Koenigs function} and its image $\Omega=h(\mathbb D)$ is the {\it Koenigs domain} of the semigroup. The Koenigs domain is {\it convex in the positive direction} (or {\it starlike at infinity}); that is, it has the property that if $z\in\Omega$, then the half-line $\{z+t:t\geq 0\}$ is contained in $\Omega$. 

\medskip

Conversely, suppose that $\Omega\neq \C$ is a domain which is convex in the positive direction and contains the origin. Then the positive real semiaxis is a slit in 
$\Omega$ landing at $\infty$. Denote by  $P_\infty$ the  prime end of $\Omega$ to which the slit converges, and
consider the Riemann mapping $h:\D\to\Omega$ with $h(0)=0$ and $h(1)=P_\infty$
(in the sense of Carath\'eodory boundary correspondence). Then the family of functions
$$
\phi_t(z)=h^{-1}(h(z)+t),\;\;\;\;\;z\in \mathbb D,\;t\geq 0,
$$
forms a semigroup with Denjoy-Wolff point $\tau=1$ and Koenigs domain $\Omega$.
The theory of semigroups is the interplay of geometric properties of $\Omega$, dynamical properties of $(\phi_t)_{t\geq 0}$, and analytic properties of the Koenigs function $h$ (and of the infinitesimal generator of the semigroup, which is not needed in the present work).

\medskip

For $z\in\mathbb D$, the curve
\begin{equation}\label{i3}
\gamma_z=\{\phi_t(z):\;t\geq 0\}\subset \mathbb D
\end{equation} 
is the {\it orbit} (or {\it trajectory}) starting from $z$. Note that the image of $\gamma_z$ under $h$ is the half-line $\{h(z)+t:t\geq 0\}$; thus, the Koenigs function linearizes the orbits. 

\medskip

We will now use some simple facts from the hyperbolic geometry of the unit disk. We denote by $\rho_{\mathbb D}(z,w)$ the hyperbolic distance of the points $z,w\in \mathbb D$; (the definition and main properties of the hyperbolic distance are briefly presented in Section 2). The diameter $(-\tau,\tau)$ of $\mathbb D$ is a hyperbolic geodesic. If $t\geq 0$, then $\phi_t(0)$ is a point on the orbit $\gamma_0$ starting from the origin. Let $\pi(\phi_t(0))$ be the hyperbolic projection of $\phi_t(0)$ on $(-\tau,\tau)$; that is, $\pi(\phi_t(0))$ is the point on 
 $(-\tau,\tau)$ of least hyperbolic distance from $\phi_t(0)$.
 
 \medskip
 
Bracci \cite{Bra} introduced three functions (speeds) associated to a semigroup  $(\phi_t)$:
\begin{eqnarray}\label{i4}
v(t)&=&\rho_{\mathbb D}(0,\phi_t(0)), \;\;\;t\geq 0 \;\;\;\hbox{(total speed)},\\
 \label{i5}
v^o(t)&=&\rho_{\mathbb D}(0,\pi(\phi_t(0))), \;\;\;t\geq 0 \;\;\;\hbox{(orthogonal speed)},\\
\label{i6}
v^T(t)&=&\rho_{\mathbb D}(\phi_t(0),\pi(\phi_t(0))), \;\;\;t\geq 0 \;\;\;\hbox{(tangential speed)}.
\end{eqnarray}  
Bracci showed that the study of these speeds reveals several deep properties of the semigroup and posed some problems about the speeds. One of them is the following question (stated in our words).

{\bf Question 4} in \cite{Bra}: Suppose that $(\phi_t)$, $(\widetilde{\phi}_t)$ are 
semigroups with Denjoy-Wolff points $\tau,\widetilde{\tau}\in\partial \mathbb D$, Koenigs domains $\Omega,\widetilde{\Omega}$, and orthogonal speeds 
$v^o, \widetilde{v}^o$, respectively. Is it true that if $\Omega\subset\widetilde{\Omega}$, then
\begin{equation}\label{i7}
\liminf_{t\to\infty}[v^o(t)-\widetilde{v}^o(t)]>-\infty?
\end{equation} 
Bracci posed this question for parabolic semigroups beacause for other cases the answer is known. Our approach to this problem covers all cases.
It is observed in \cite{BCK} that (\ref{i7}) implies the existence of a constant $K>0$ such that for every $t\geq 0$,
\begin{equation}\label{i8}
|\phi_t(0)-\tau|\leq K\;|\widetilde{\phi}_t(0)-\widetilde{\tau}|.
\end{equation} 
Thus Question 4 asks whether a geometric property regarding the Koenigs domains 
$\Omega$ and $\widetilde{\Omega}$ (namely, the inclusion $\Omega\subset\widetilde{\Omega}$) implies a dynamic property (namely, that  $(\widetilde{\phi}_t)$ converges to $\widetilde{\tau}$ slower than 
$(\phi_t)$ converges to $\tau$). 

\medskip

Question 4 was studied by Bracci, Cordella, and Kourou \cite{BCK}. They gave a positive answer for a wide class of pairs of semigroups $(\phi_t)$, $(\widetilde{\phi}_t)$. They also proved that the answer is positive if the function
$\widetilde{v}^o$ is eventually increasing; (we use the term {\it increasing} in the weak sense, and use the term {\it strictly increasing} for the strong monotonicity).
Thus they were led to the following question.

{\bf Question (ii)} in \cite{BCK}: Does there exist a non-elliptic semigroup such that the orthogonal speed is not eventually increasing? 

We give a negative answer to this question:
\begin{theorem}\label{Th1}
Let $(\phi_t)$ be a non-elliptic semigroup in $\D$ with orthogonal speed $v^o$. The function $v^o$ is strictly increasing in $[0,+\infty)$.	
\end{theorem}
It follows that the answer to Question 4 is positive. Theorem \ref{Th1} has a simple geometric interpretation: Suppose for simplicity that $\tau=1$. For $x\in (0,1)$, let $\Gamma_x$ be the hyperbolic geodesic for $\D$ which is perpendicular to the real axis at the point $x$. Recall that $\Gamma_x$ is a circular arc orthogonal to the unit circle. By Theorem \ref{Th1}, for any $x\in (0,1)$, the orbit $\gamma_0$ intersects $\Gamma_x$ at exactly one point. Thus, if this orbit enters the hyperbolic halfplane bounded by $\Gamma_x$ and containing a neighborhood of $1$, then it never exits it. We remark that by Julia's Lemma, an analogous monotonicity property holds for horodisks instead of halfplanes.

\medskip

 One can define the orthogonal speed using any point in the unit disk (see \cite[Definition 3.5]{Bra}). Fix $z\in\D$ and let $\gamma$ be the hyperbolic geodesic having one end point at $1$, and such that $z\in\gamma$. For $t>0$, we denote by $\pi_{\gamma}(\phi_t(z))$ the hyperbolic projection of $\phi_t(z)$ on $\gamma$; that is, $\pi_{\gamma}(\phi_t(z))$ is the point on $\gamma$ having least hyperbolic distance from $\phi_t(z)$.  Let
\begin{equation}\label{genspeed}
v^{z}(t)=\rho_{\D}\left(z,\pi_{\gamma}(\phi_t(z))\right).
\end{equation}
The proof of Theorem \ref{Th1} does not suffer if we replace $v^o$ by $v^z$ and thus $v^{z}(t)$ is also a strictly increasing function of $t>0$, for any $z\in\D$.

\medskip

The article  \cite{BCK} contains also the following result: 
\\
{\it Suppose that $(\phi_t)$ is a non-elliptic semigroup. If the total speed is (eventually) increasing, then the orthogonal speed is also (eventually) increasing.}\\
In view of this result, one may conjecture that the total speed is an eventually increasing function. We show that this conjecture is false.
\begin{theorem}\label{Th2}
There exists a non-elliptic semigroup such that its total speed $v$ is not an eventually increasing function.
\end{theorem}

\medskip

In Section 3, we will prove a result concerning the asymptotic behavior of the total speed.
\begin{theorem}\label{Th3}
There exist a positive number $\alpha<1$ and a non-elliptic semigroup with total speed $v$ such that
\[
\limsup_{t\to +\infty}\frac{v(t)}{t^\alpha}=+\infty\;\;\;\; \hbox{and}\;\;\;\;
\liminf_{t\to +\infty}\frac{v(t)}{t^\alpha}=0.
\]
\end{theorem}
This theorem provides an answer to another question of Bracci\footnote{Bracci and Cordella informed us that they have given an independent answer to that question.} (Question 5 in \cite{Bra}). 

In the last section of the paper we will use certain harmonic measure estimates to give an explicit lower bound for \eqref{i7}:
\begin{theorem}\label{Th4}
Suppose that $(\phi_t)$, $(\widetilde{\phi}_t)$ are 
semigroups with Denjoy-Wolff points $\tau,\widetilde{\tau}\in\partial \mathbb D$, Koenigs domains $\Omega,\widetilde{\Omega}$, and orthogonal speeds 
$v^o, \widetilde{v}^o$, respectively. If $\Omega\subset\widetilde{\Omega}$, then
\[
\liminf_{t\to\infty}[v^o(t)-\widetilde{v}^o(t)]\geq -\log 2.
\]
\end{theorem}

It is possible to generalize this bound for the orthogonal speeds with starting point any $z\in\D$, as defined in \eqref{genspeed}. Write
\[
v^z(t)-\widetilde{v}^z(t)=v^z(t)-v^o(t)+v^o(t)-\widetilde{v}^o(t)+\widetilde{v}^o(t)-\widetilde{v}^z(t).
\]
The term $v^o(t)-\widetilde{v}^o(t)$ can be estimated using Theorem \ref{Th4} while the other two terms, $v^z(t)-v^o(t)$ and $\widetilde{v}^o(t)-\widetilde{v}^z(t)$, can be estimated using \cite[Lemma 3.7 (3)]{Bra}. From this it follows that the following bound holds.
\[
\liminf_{t\to\infty}[v^z(t)-\widetilde{v}^z(t)]\geq -\log 2-2\log\frac{1+|z|}{1-|z|}.
\]
%%%%%%%%%%%%%%%%%%%%%%%%%%%%%%%%%%%%%%%%%%%%%%%%%%%%%%%%%%%%%%%%%%%%%%%%%%%%%%%%%%%%%
%%%%%%%%%%%%%%%%%%%%%%%%%%%%%%%%%%%%%%%%%%%%%%%%%%%%%%%%%%%%%%%%%%%%%%%%%%%%%%%%%%%%%
%%%%%%%%%%%%%%%%%%%%%%%%%%%%%%%%%%%%%%%%%%%%%%%%%%%%%%%%%%%%%%%%%%%%%%%%%%%%%%%%%%%%%
%%%%%%%%%%%%%%%%%%%%%%%%%%%%%%%%%%%%%%%%%%%%%%%%%%%%%%%%%%%%%%%%%%%%%%%%%%%%%%%%%%%%%
%%%%%%%%%%%%%%%%%%%%%%%%%%%%%%%%%%%%%%%%%%%%%%%%%%%%%%%%%%%%%%%%%%%%%%%%%%%%%%%%%%%%%

\section{Preliminaries}
\subsection{Hyperbolic distance}
The hyperbolic metric in the unit disk is 
\[
\lambda_{\D}(z)|dz|=\frac{|dz|}{1-|z|^2}.
\]
This metric induces a distance in $\D$ in the usual way:
\[
\rho_{\D}(z,w)=\inf_{\gamma}\int_{\gamma}\lambda_{\D}(\zeta)|d\zeta|,\;\;\;z,w\in\D,
\]
where the infimum is taken over all smooth curves $\gamma$ contained in $\D$ and joining $z$ to $w$. A minimizing curve for this infimum is called a hyperbolic geodesic. For the unit disk, the hyperbolic geodesics are either radial segments or arcs of circles orthogonal to $\partial\D$. For further reading on the hyperbolic distance and related topics, we refer to \cite{BM}, \cite{BCD} as well as references therein.

Since for any conformal automorphism of the unit disk, $T$, we have the identity
\[
\frac{|T'|}{1-|T|^2}=\frac{1}{1-|z|^2},
\]
it follows that each such map $T$ is an isometry for $\rho_{\D}$. This allows us to define the hyperbolic distance in any simply connected domain which is conformally equivalent to $\D$. If $\Omega$ is such a domain, let $f:\D\to\Omega$ be a conformal map and define for $z,w\in\Omega$,
\[
\rho_{\Omega}(z,w)=\rho_{\D}(f^{-1}(z),f^{-1}(w)),
\]
and 
\[
\lambda_{\Omega}(z)=\frac{\lambda_{\D}(f^{-1}(z))}{|f'(f^{-1}(z))|}.
\]
Due to the observation above about the class of automorphisms of $\D$, these definitions do not depend on the choice of the map $f$. 

We can explicitly calculate the hyperbolic distance in $\D$:
\[
\rho_{\D}(z,w)=\frac{1}{2}\log\frac{1+|T(z,w)|}{1-|T(z,w)|},\ z,w\in\D,
\]
where $T(z,w)=\frac{z-w}{1-z\overline{w}}$. Using a conformal map from the unit disk onto the upper half plane, $\Ha$, it can be shown that
\[
\lambda_{\Ha}(z)|dz|=\frac{|dz|}{2\Im z},
\]
and in a similar way, if $P$ is any open half plane, then
\begin{equation}\label{lamb}
\lambda_P(z)|dz|=\frac{|dz|}{2\dist(z,\partial P)}.
\end{equation}
See, for example, \cite{BM}. Using the Schwarz-Pick Lemma, it can be shown that the hyperbolic distance decreases when the domain gets larger, i.e., if $D_1\subset D_2$ are two simply connected domains, then 
\[
\rho_{D_1}(z,w)\geq \rho_{D_2}(z,w),
\]
for every $z,w\in D_1$. 

The quasihyperbolic distance in a simply connected domain $\Omega$, from $z\in\Omega$ to $w\in\Omega$, is 
\[
Q_{\Omega}(z,w)=\inf_{\gamma}\int_{\gamma}\frac{|d\zeta|}{\dist(\zeta,\partial\Omega)},
\]
where the infimum is taken over all smooth curves in $\Omega$ joining $z$ to $w$. Even though this is not a conformally invariant quantity, it can be used to estimate the hyperbolic distance in $\Omega$:
\begin{equation}\label{qh}
\frac{1}{4}Q_{\Omega}(z,w)\leq \rho_{\Omega}(z,w)\leq Q_{\Omega}(z,w),
\end{equation}
for all $z,w\in\Omega$. For a proof, see \cite{BM}.

We now state two key facts that we shall need.
\begin{lemma}\label{PL1}
Suppose that $H$ is any disk or half-plane. Then for all $z,w\in H$, 
\[
\sinh^2\left(\rho_{H}(z,w)\right)=|z-w|^2\lambda_{H}(z)\lambda_{H}(w).
\]
\end{lemma}
For a proof, see \cite[Theorem 7.4]{BM}.
\begin{lemma}\label{PL2}
Suppose that $\Omega_n$, $n\in\N$, is a sequence of simply connected domains, containing $0$, which converges in the sense of Carath\'eodory to a simply connected domain $\Omega_0\neq\C$ which contains $0$. Let $K$ be a compact set contained in each $\Omega_n$, $n\geq 0$. Suppose that there exists a simply connected domain $K'$ which contains $K$ and is contained in $\Omega_n$, for all $n\geq 0$. Then
for any fixed $\zeta\in K$,
\begin{equation}\label{h1}
\sup_{w\in K}\lvert \rho_{\Omega_n}(w,\zeta)-\rho_{\Omega_0}(w,\zeta)\rvert\to 0,
\end{equation}
as $n\to +\infty$. In particular, for all $a,b\in K$,
\begin{equation}\label{h2}
\rho_{\Omega_n}(a,b)\to\rho_{\Omega_0}(a,b),\ \hbox{as} \ n\to +\infty.
\end{equation}
.
\end{lemma}
\begin{proof}
We first prove \eqref{h2}. Let $\Omega_n$, $n\geq 0$ be as in the statement of the Lemma and suppose that $\Omega_n\to\Omega_0$, as $n\to +\infty$, in the sense of Carath\'eodory. For each $n\in\N$, let $f_n$ be the conformal map from $\D$ onto $\Omega_n$ satisfying $f_n(0)=0$ and $f_n'(0)>0$. By the kernel convergence theorem, the sequence $f_n$ converges locally uniformly, in $\D$, to a conformal map $f$ from $\D$ onto $\Omega_0$. Let $K$ satisfy the hypotheses of the Lemma and fix $a,b\in K$. For each $n\geq 0$, let $g_n=f_n^{-1}:\Omega_n\to\D$. Note that by conformal invariance and the triangle inequality,
\begin{align*}
\lvert\rho_{\Omega_n}(a,b)-\rho_{\Omega_0}(a,b)\rvert &= \lvert \rho_{\D}(g_n(a),g_n(b))-\rho_{\D}(g_0(a),g_0(b))\rvert \\
&\leq \rho_{\D}(g_n(a),g_0(a))+\rho_{\D}(g_n(b),g_0(b)).
\end{align*}
By \cite[Theorem 3.5.8]{BCD}, $g_n(w)$ tends to $g_0(w)$, for any $w\in K$ and thus \eqref{h2} holds. 

We will now prove \eqref{h1}. For each $n$, let $w_n$ be the point of $K$ where the supremum is attained . By passing to a subsequence, if needed, we may assume that $w_n$ converges to some point $w_0\in K$. Then
\begin{align*}
\lvert \rho_{\Omega_n}(w_n,\zeta)-\rho_{\Omega_0}&(w_n,\zeta)\rvert \leq \lvert\rho_{\Omega_n}(w_n,\zeta)-\rho_{\Omega_n}(w_0,\zeta)\rvert +\\
&+\lvert\rho_{\Omega_n}(w_0,\zeta)-\rho_{\Omega_0}(w_0,\zeta)\rvert +
\lvert\rho_{\Omega_0}(w_0,\zeta)-\rho_{\Omega_0}(w_n,\zeta)\rvert .
\end{align*}
Let $K'$ be a simply connected domain containing $K$ and contained in $\Omega_n$, for all $n\geq 0$. Then note that, by the triangle inequality and the domain monotonicity property for the hyperbolic distance, we have
\[
\lvert\rho_{\Omega_n}(w_n,\zeta)-\rho_{\Omega_n}(w_0,\zeta)\rvert\leq \rho_{\Omega_n}(w_n,w_0)\leq \rho_{K'}(w_n,w_0)\to 0,
\]
as $n\to +\infty$, since $w_n\to w_0$ and they both stay away from $\partial K'$. Using the same argument, we see that $\lvert\rho_{\Omega_0}(w_0,\zeta)-\rho_{\Omega_0}(w_n,\zeta)\rvert$ tends to $0$ as well. Finally, by \eqref{h2}, $\lvert\rho_{\Omega_n}(w_0,\zeta)-\rho_{\Omega_0}(w_0,\zeta)\rvert\to 0$. The conclusion follows.
\end{proof}

\subsection{Harmonic measure. Basic properties.}
Let $\Omega\neq\C$ be a simply connected domain and let $E\subset\partial\Omega$ be a Borel set. The harmonic measure of $E$ with respect to $\Omega$ is the Perron solution, $u$, of the Dirichlet problem for the Laplacian in $\Omega$ with boundary function equal to $1$ on $E$ and $0$ on $\partial\Omega\setminus E$. We will use the standard notation $u(z)=\omega(z,E,\Omega)$, $z\in\Omega$. We refer to \cite{BCD}, \cite{Ran}  for  presentations of the theory of harmonic measure.

\medskip

We review some basic properties of harmonic measure. We start with
 its domain
monotonicity: If $\Omega_1\subset\Omega_2$, $E\subset \partial
\Omega_1\cap\partial\Omega_2$, and $z\in \Omega_1$, then
\begin{equation}\label{h10}
\omega(z,E,\Omega_1)\leq \omega(z,E,\Omega_2).
\end{equation}
There is a more precise statement which may be called the strong
Markov property of harmonic measure (see \cite[p.117]{Doo}): If
$\Omega_1\subset\Omega_2$, $E\subset \partial\Omega_2$,
and $z\in \Omega_1$, then
\begin{equation}\label{h11}
\omega(z,E,\Omega_2)=\omega(z,E\cap \partial \Omega_1,\Omega_1)+\int_{\Omega_2\cap \partial \Omega_1}
\omega(z,d\zeta,\Omega_1)\;\omega(\zeta,E,\Omega_2).
\end{equation}

Another property of harmonic measure  is that it
is invariant under conformal mappings; see
\cite[Th. 4.3.8]{Ran}.
We will also need the following reflection property. Here and below, the superscript $*$ denotes reflection in the real axis.

\begin{lemma}\label{RL}
Let $D$ be a  simply connected domain in the upper half plane and assume that the set $\partial D\cap\mathbb R$ is an interval $I$. Set $\delta=\partial D\setminus I$. Let $G$ be a simply connected domain in the lower half plane such that $\partial G\cap\mathbb R=I$. Assume  that $G^*\subset D$. Consider the simply connected domain $\Omega:=D\cup I\cup G$. \\
{\rm (a)} For every $z\in G$ and every Borel subset $B$ of $I$, $\omega(z,B,G)\leq \omega(\bar{z},B,D)$.\\
{\rm (b)} For every $z\in G\cup I$, $\omega(z,\delta^*\cap\partial G,\Omega)\leq \omega(\bar{z},\delta,\Omega)$.\\
{\rm (c)} For every $z\in G$ and every Borel subset $\sigma$ of $\partial G\setminus I$ satisfying $\sigma^*\subset \delta$, we have $\omega(z,\sigma,G)\leq\omega(\bar{z},\delta,D)$.
\end{lemma}
\begin{proof}
The proofs of (a) and (c) are elementary. By symmetry and domain monotonicity,
$$
\omega(z,B,G)=\omega(\bar{z},B,G^*)\leq \omega(\bar{z},B,D),
$$
which proves (a). Similarly, if $\sigma$ is as in the statement (c), then
$$
\omega(z,\sigma,G)=\omega(\bar{z},\sigma^*,G^*)\leq \omega(\bar{z},\sigma^*,D)\leq \omega(\bar{z},\delta,D).
$$
Part (b) follows at once from a deep polarization result of A.Yu.Solynin \cite[Theorem 2]{Sol}.	
\end{proof}

\subsection{Harmonic measure and hyperbolic geodesics}

Let $\Omega\neq\C$ be a simply connected domain and let $f:\D\to\Omega$ be a conformal mapping of $\D$ onto $\Omega$. We denote by $\hat{f}$ the Carath\'eodory extension of $f$ on the boundary of $\D$. Thus $\hat{f}(\partial\D)$ is the set of prime ends of $\Omega$; we denote this set by $\partial_C\Omega$. If $p,q\in\partial_C\Omega$, we denote by $[q,p]_{\Omega}$ the hyperbolic geodesic for $\Omega$ which joins $q$ to $p$. Note that by definition, $[q,p]_{\Omega}=\hat{f}(\gamma)$, where $\gamma$ is the hyperbolic geodesic in the unit disk joining $\hat{f}^{-1}(q)$ to $\hat{f}^{-1}(p)$. If $a,b\in\Omega\cup\partial_C\Omega$, we denote by $[a,b]_{\Omega}$ the arc of the hyperbolic geodesic joining $a$ to $b$. We also define $(a,b)_{\Omega}=[a,b]_{\Omega}\setminus\{a,b\}$.

Suppose now that $q,p\in\partial_C\Omega$ and $\gamma=(q,p)_{\Omega}$ is the hyperbolic geodesic for $\Omega$ joining $q$ to $p$. Then $\gamma$ divides $\Omega$ into two subdomains $\Omega_{-}$ and $\Omega_{+}$. Let $z_0\in\Omega_{-}$. Let $p(z_0)$ be the point on $\gamma$ having minimal hyperbolic distance from $z_0$, i.e.,
\[
\rho_{\Omega}\left(z_0,p(z_0)\right)=\min_{\zeta\in\gamma}\rho_{\Omega}\left(z_0,\zeta\right).
\]
By mapping $\Omega$ conformally onto $\D$ so that $z_0$ corresponds to $0$, it is not hard to see that the geodesic $[z_0,p(z_0)]_{\Omega}$ meets $\gamma$ perpendicularly at $p(z_0)$.

Now choose a conformal map $\phi:\D\to\Omega$ from $\D$ onto $\Omega$ so that $\phi\left((-1,1)\right)=\gamma$ and $\phi^{-1}(z_0)=-it_0$, for some $t_0\in (0,1)$. Observe that the point in $(-1,1)$ having minimal hyperbolic distance, in $\D$, from $-it_0$ is $0$ and the geodesic joining these two points is the line segment $l:=[-it_0,0]$. By conformal invariance, it follows that $\phi(0)=p(z_0)$ and $\phi(l)=[z_0,p(z_0)]_{\Omega}$. Observe that $l$ is a hyperbolic geodesic in the lower semidisk $\D_{-}:=\D\cap\{z:\ \Im z<0\}$ as well. Since the restriction of $\phi$ on $\D_{-}$ is a conformal mapping onto $\Omega_{-}$, it follows that $[z_0,p(z_0)]_{\Omega}$ is also a geodesic for the domain $\Omega_{-}$, i.e.,
\[
[z_0,p(z_0)]_{\Omega}=[z_0,p(z_0)]_{\Omega_{-}}.
\]

We can now prove the following two needed results.
\begin{lemma}\label{PL3}
Let $\Omega$, $p$, $q$, $\gamma$, $z_0$, and $p(z_0)$ be as in the discussion above. Then
\[
\omega\left(z_0,(q,p(z_0)]_{\Omega},\Omega_{-}\right)=\omega\left(z_0,[p(z_0),p)_{\Omega},\Omega_{-}\right)=\frac{1}{2}\omega\left(z_0,\gamma,\Omega_{-}\right).
\]
\end{lemma}
\begin{proof}
As we observed above, $\phi$ maps conformally $\D_{-}$ onto $\Omega_{-}$ so that $\phi(-it_0)=z_0$ and $\phi(0)=p(z_0)$. The conformal invariance for the harmonic measure immediately implies the desired result.
\end{proof}
\begin{lemma}\label{PL4}
Let $\Omega$, $p$, $q$, $\gamma$, $z_0$, and $p(z_0)$ be as in the discussion above. Let $q_0\in\partial_C\Omega$ be the prime end such that $[z_0,p(z_0)]_{\Omega}\subset [q_0,p(z_0)]_{\Omega}$, i.e. $\hat{\phi}(-i)=q_0$. Let $\gamma_0=[q_0,p(z_0)]_{\Omega}$ and $\omega(z)=\omega\left(z,\gamma,\Omega_{-}\right)$, $z\in\Omega_{-}$. Then
\[
\nabla \omega(z_0)=\lambda\epsilon(z_0),
\]
for some $\lambda>0$, where $\epsilon(z_0)$ is the unit tangent vector of the curve $\gamma_0$ at the point $z_0$; ($\gamma_0$ has the orientation from $q_0$ to $p(z_0)$).
\end{lemma}
\begin{proof}
Let $g=\phi^{-1}$. By the preceding discussion, $g$ maps $\Omega_{-}$ conformally onto $\D_{-}$ so that $g(\gamma_0)$ is the line segment $[-i,0]$,  $g(z_0)=-it_0$, and $g(p(z_0))=0$. Note that the curve $\gamma_0$ can be parametrized as $\gamma_0(t)=\phi(it)$, $t\in [-1,0]$. Then $\epsilon(z_0)$ is a positive multiple of the vector $\gamma_0'(-t_0)$. We calculate
\[
\gamma_0'(-t_0)=i\phi'(-it_0)=\frac{i\overline{g'(z_0)}}{|g'(z_0)|^2}.
\]
Therefore, $\epsilon(z_0)$ is in the same direction as the vector 
\[
\delta(z_0):=\left( (\Im g)_x(z_0), (\Re g)_x(z_0)\right),
\] 
where the subscripts denote partial derivatives. By conformal invariance, for $z\in\Omega_{-}$,
\[
\omega(z)=\omega\left(g(z),(-1,1),\D_{-}\right)=v\circ g(z),
\]
where we set $v(z)=\omega\left(z,(-1,1),\D_{-}\right)$, $z\in\D_{-}$.
By the chain rule, 
\[
\omega_x(z_0)=v_x(-it_0)(\Re z)_x(z_0)+v_y(-it_0)(\Im g)_x(z_0),
\]
and
\[
\omega_y(z_0)=v_x(-it_0)(\Re g)_y(z_0)+v_y(-it_0)(\Im g)_y(z_0).
\]
By the symmetry of $\D_{-}$, it is not hard to see that $v_x(-it_0)=0$. Hence, by the Cauchy-Riemann equations, 
\[
\nabla\omega(z_0)=v_y(-it_0)\delta(z_0).
\]
Since $v_y(-it_0)>0$, the conclusion follows.
\end{proof}

%%%%%%%%%%%%%%%%%%%%%%%%%%%%%%%%%%%%%%%%%%%%%%%%%%%%%%%%%%%%%%%%%%%%%%%%%%%%%%%%%%%%%%%%%
%%%%%%%%%%%%%%%%%%%%%%%%%%%%%%%%%%%%%%%%%%%%%%%%%%%%%%%%%%%%%%%%%%%%%%%%%%%%%%%%%%%%%%%%%
%%%%%%%%%%%%%%%%%%%%%%%%%%%%%%%%%%%%%%%%%%%%%%%%%%%%%%%%%%%%%%%%%%%%%%%%%%%%%%%%%%%%%%%%%
%%%%%%%%%%%%%%%%%%%%%%%%%%%%%%%%%%%%%%%%%%%%%%%%%%%%%%%%%%%%%%%%%%%%%%%%%%%%%%%%%%%%%%%%%
%%%%%%%%%%%%%%%%%%%%%%%%%%%%%%%%%%%%%%%%%%%%%%%%%%%%%%%%%%%%%%%%%%%%%%%%%%%%%%%%%%%%%%%%%

\section{An example for the asymptotic behavior of the total speed}

We now proceed with several lemmas which are needed for the proof of Theorem \ref{Th3}.
We construct a convex in the positive direction simply connected domain $\Omega$ as follows. For integer $n\geq 0$, let $t_n=2^{2^n}$. For $n\geq 1$, consider the rectangles 
\[
R_n=\left\{
\begin{array}{ll}
\left( t_{n-1},t_n \right)\times \left(-t_n^{1/3},t_n^{1/3}\right), & \hbox{if}\ n\ \hbox{is\ even,}\\
\left(t_{n-1},t_n \right)\times \left(-t_n^{1/2},t_n^{1/2}\right), & \hbox{if}\ n\ \hbox{is\ odd,}\\

\end{array}
\right.
\] 
and for $n=0$, let
\[
R_0=\{z\in\C:\ -1<\Im z<1,\ \Re z<2\}.
\]
We set $\Omega=\left(\bigcup_{j\geq 0}\overline{R_j}\right)^{\circ}$. See Figure \ref{Fig1}. For two sequences of real numbers $x_n$, $y_n$, we will use the notation $x_n\gtrsim y_n$ to indicate that there exists a universal constant $c>0$ such that $x_n\geq cy_n$, for all $n\in\N$. The notation $x_n\asymp y_n$ is equivalent to $x_n\gtrsim y_n$ and $y_n\gtrsim x_n$.

\begin{lemma}\label{L1.1}
Consider the sequence
\[
a_n=\left\{
\begin{array}{ll}
\frac{1}{2}, & \hbox{if}\ n\ \hbox{is\ odd,}\\
\\
\frac{1}{3}, & \hbox{if}\ n\ \hbox{is\ even.}\\
\end{array}
\right.
\]
Then, for $n$ large, we have
\[
\rho_{\Omega}(t_{n-1},t_n)\asymp t_n^{1-a_n}.
\]

\end{lemma}
\begin{figure}
	\includegraphics[width=1\linewidth]{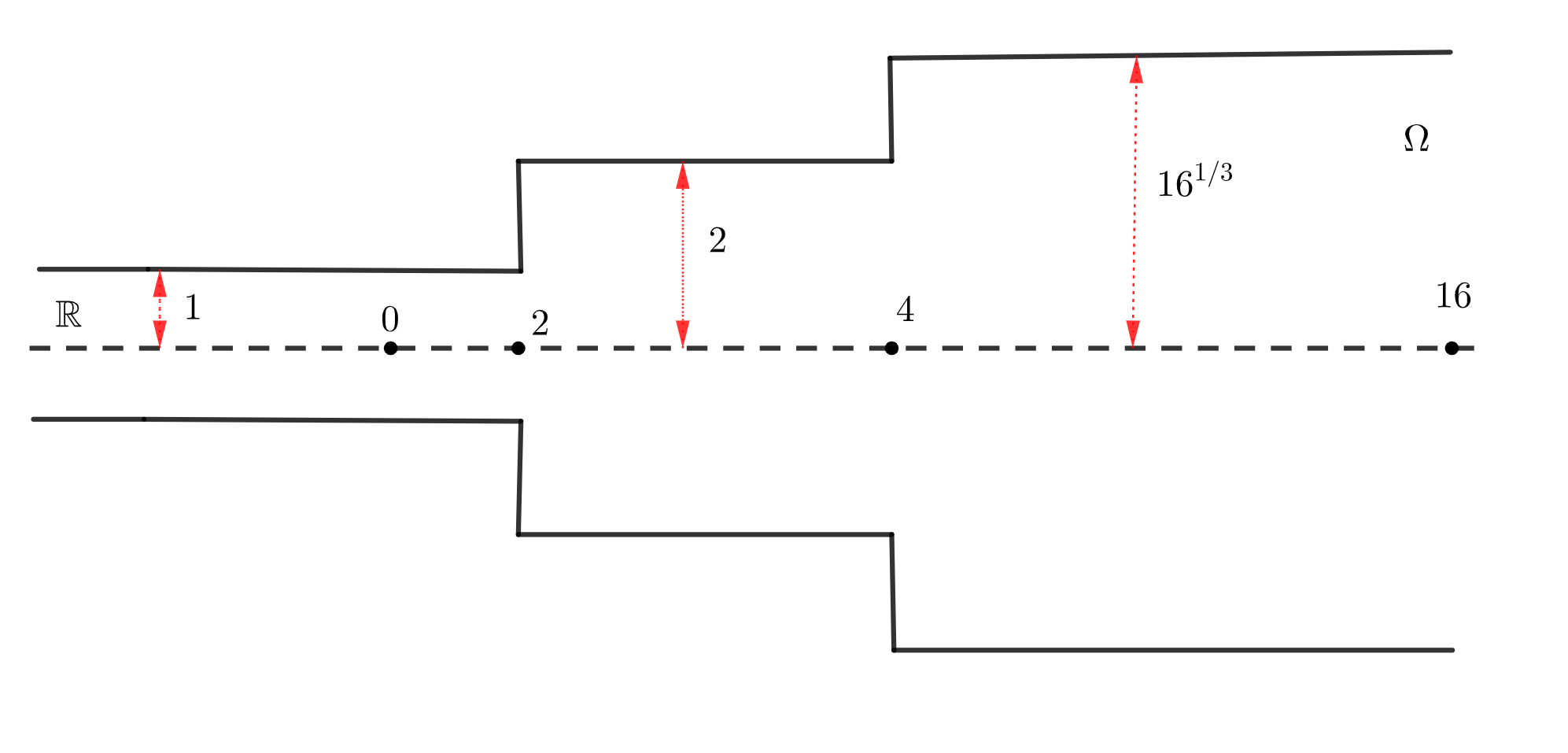}
	\caption{The domain $\Omega$.}
	\label{Fig1}	
\end{figure}
\begin{proof}
Since the domain $\Omega$ is symmetric with respect to $\R$, it follows that $\R$ is a hyperbolic geodesic for $\Omega$. In this case, the infimum for the quasihyperbolic distance, $Q_{\Omega}(a,b)$, is attained at the segment $[a,b]$, for $a,b\in\R$. Therefore, by \eqref{qh}, for $n\in\N$,
\begin{equation}\label{E1.1.1}
\rho_{\Omega}(t_{n-1},t_n)\gtrsim \int_{t_{n-1}}^{t_n}\frac{dw}{\dist(w,\partial\Omega)}\geq\frac{t_n-t_{n-1}}{t_n^{a_n}}\gtrsim t_n^{1-a_n}.
\end{equation}
The last inequality follows from the choice of the sequence $t_n$.

Moreover, if $n$ is large enough, then for $t\in [t_{n-1}+t_n^{a_n},t_n]$, we have $\dist(t,\partial\Omega)=t_n^{a_n}$. It follows, again by \eqref{qh}, that
\begin{equation}\label{E1.1.2}
\rho_{\Omega}(t_{n-1}+t_n^{a_n},t_n)\lesssim \int_{t_{n-1}+t_n^{a_n}}^{t_n}\frac{dw}{\dist(w,\partial\Omega)}=\frac{t_n-t_{n-1}-t_n^{a_n}}{t_n^{a_n}}\leq t_n^{1-a_n}.
\end{equation}
Similarly, 
$$
\rho_{\Omega}(t_{n-1},t_{n-1}+t_n^{a_n})\lesssim \frac{t_n^{a_n}}{t_{n-1}^{a_{n-1}}}.
$$
Note that $t_n=t_{n-1}^2$ and $a_n-\frac{a_{n-1}}{2}\leq 1-a_n$. Therefore, the last estimate becomes
\begin{equation}\label{E1.1.3}
\rho_{\Omega}(t_{n-1},t_{n-1}+t_n^{a_n})\lesssim t_n^{1-a_n}.
\end{equation}
By the triangle inequality, \eqref{E1.1.2}, \eqref{E1.1.3} and \eqref{E1.1.1}, we easily arrive at the desired result.
\end{proof}
\begin{lemma}\label{L1.2}
For $n$ sufficiently large, we have
\[
\rho_{\Omega}(0,t_n)\asymp \rho_{\Omega}(t_{n-1},t_n).
\]
\end{lemma}
\begin{proof}
As we have observed previously, $\R$ is a hyperbolic geodesic for $\Omega$ and thus
\[
\rho_{\Omega}(t_{n-1},t_n)<\rho_{\Omega}(0,t_n),
\]
for all $n\in\N$.
Moreover, for any $n\in\N$, 
\[
\rho_{\Omega}(0,t_n)=\rho_{\Omega}(0,t_{n-1})+\rho_{\Omega}(t_{n-1},t_n).
\]
Hence, it suffices to show that 
\begin{equation}\label{E1.2.1}
\rho_{\Omega}(0,t_{n-1})\lesssim \rho_{\Omega}(t_{n-1},t_n),
\end{equation}
for $n$ large enough.
We will prove \eqref{E1.2.1} when $n$ is odd. The case when $n$ is even can be handled in a similar manner. Observe that 
\[
\rho_{\Omega}(0,t_{n-1})=\rho_{\Omega}(0,2)+\sum_{j=1}^{n-1}\rho_{\Omega}(t_{j-1},t_j).
\]
Therefore, by Lemma \ref{L1.1}, for $n$ sufficiently large,
\begin{align*}
\rho_{\Omega}(0,t_{n-1})&\lesssim \sum_{j=1}^{n-1}t_j^{1-a_j}\leq \sum_{j=1}^{n-1}t_j^{\frac{2}{3}}\leq (n-1)t^{\frac{2}{3}}_{n-1}=(n-1)2^{\frac{2}{3}2^{n-1}}\\
& \lesssim 2^{2^{n-1}}=t_n^{1-a_n}\asymp \rho_{\Omega}(t_{n-1},t_n).
\end{align*}
This completes the proof.
\end{proof}

We can now prove Theorem \ref{Th3}, which we state here in an equivalent form.
\begin{theorem}\label{T1}
There exist a positive number $\alpha<1$ such that
\[
\limsup_{t\to +\infty}\frac{\rho_{\Omega}(0,t)}{t^\alpha}=+\infty\;\;\;\;\hbox{and}\;\;\;\;
\liminf_{t\to +\infty}\frac{\rho_{\Omega}(0,t)}{t^\alpha}=0,
\]
where $\Omega$ is the domain defined above.
\end{theorem}
\begin{proof}
Let $t_n=2^{2^n}$. By Lemma \ref{L1.2} and Lemma \ref{L1.1}, for $n$ large enough, we have
\[
\rho_{\Omega}(0,t_n)\asymp \rho_{\Omega}(t_{n-1},t_n)\asymp t_n^{1-a_n}=
\left\{
\begin{array}{ll}
t_n^{\frac{1}{2}}, & \hbox{if}\ n\ \hbox{is\ odd,}\\
\\
t_n^{\frac{2}{3}}, & \hbox{if}\ n\ \hbox{is\ even.}\\
\end{array}
\right.
\]
Then it is easy to check that
\[
\limsup_{n\to +\infty} \frac{\rho_{\Omega}(0,t_n)}{t^{\frac{7}{12}}}=\lim_{n\to +\infty} t_n^{\frac{1}{12}}=+\infty,
\]
and
\[
\liminf_{n\to +\infty} \frac{\rho_{\Omega}(0,t_n)}{t^{\frac{7}{12}}}=\lim_{n\to +\infty} t_n^{-\frac{1}{12}}=0.
\]
\end{proof}

%%%%%%%%%%%%%%%%%%%%%%%%%%%%%%%%%%%%%%%%%%%%%%%%%%%%%%%%%%%%%%%%%%%%%%%%%%%%%%%%%%%%%%%%%
%%%%%%%%%%%%%%%%%%%%%%%%%%%%%%%%%%%%%%%%%%%%%%%%%%%%%%%%%%%%%%%%%%%%%%%%%%%%%%%%%%%%%%%%%
%%%%%%%%%%%%%%%%%%%%%%%%%%%%%%%%%%%%%%%%%%%%%%%%%%%%%%%%%%%%%%%%%%%%%%%%%%%%%%%%%%%%%%%%%
%%%%%%%%%%%%%%%%%%%%%%%%%%%%%%%%%%%%%%%%%%%%%%%%%%%%%%%%%%%%%%%%%%%%%%%%%%%%%%%%%%%%%%%%%
%%%%%%%%%%%%%%%%%%%%%%%%%%%%%%%%%%%%%%%%%%%%%%%%%%%%%%%%%%%%%%%%%%%%%%%%%%%%%%%%%%%%%%%%%

\section{The total speed is not always increasing}
In this section, we will construct a convex in the positive direction, simply connected domain $\tilde{\Omega}$ containing $\R$ such that 
\[
\rho_{\tilde{\Omega}}(0,x_n)>\rho_{\tilde{\Omega}}(0,y_n),
\]
for some sequences $x_n,\ y_n\in\R$, tending to $+\infty$ and satisfying $x_n<y_n<x_{n+1}$, for all $n$. We will begin by constructing a convenient sequence of domains $\Omega_n$, $n\geq 0$.
\begin{lemma}\label{L2.1}
There exists an absolute constant $\delta>0$ and two sequences of real numbers $a_n,\ b_n$, $n\geq 0$, tending to infinity, as $n\to +\infty$, and satisfying $a_n+b_n<a_{n+1}-b_{n+1}$, $n\geq 0$, such that if 
\[
\Omega_n=\C\setminus\bigcup_{k=0}^{n}\{z:\ \Re z\leq a_k,\ \Im z= -b_k\},\]
then
\begin{equation}\label{omgn}
\rho_{\Omega_n}(0,a_k-b_k)-\rho_{\Omega_n}(0,a_k+b_k)\geq \delta,
\end{equation}
for all $0\leq k\leq n$.
\end{lemma}
\begin{proof}
We will define the sequences $a_n,b_n$ inductively. We begin with the case $n=0$.

Let $\Omega_0^*=\C\setminus\{z:\ \Re z\leq 0,\ \Im z=-1\}$. Let $K=\partial D(-i,R)\cap \{z:\Im z\geq 0\}$, where $D(-i,R)$ denotes the disk centered at $-i$ of radius $R>1$. The map $\phi(z)=\sqrt{z+i}$ maps $\Omega_0^*$ conformally onto the right half plane $H$. The compact set $K$ is mapped onto a closed arc, $J$, of the circle of radius $\sqrt{R}$ centered at $0$, and the points $-1$ and $1$ are mapped to $2^{1/4}e^{3\pi i/8}$ and $2^{1/4}e^{\pi i/8}$, respectively. Fix $z\in K$ and let $w=\phi(z)\in J$. We now show that we can choose $R$ sufficiently large such that
\begin{equation}\label{omg1}
\rho_{\Omega_0^*}(z,-1)-\rho_{\Omega_0^*}(z,1)>0,
\end{equation}
for all $z\in K$.  By conformal invariance, \eqref{omg1} is equivalent to
\[
\rho_{H}(w,2^{1/4}e^{3\pi i/8})>\rho_H(w,2^{1/4}e^{\pi i/8}),
\]
which in turn, via Lemma \ref{PL1}, is equivalent to
\[
\lvert w-2^{1/4}e^{3\pi i/8}\rvert^2\lambda_H(2^{1/4}e^{3\pi i/8})>
\lvert w-2^{1/4}e^{\pi i/8}\rvert^2\lambda_H(2^{1/4}e^{\pi i/8}).
\]
Since $w$ belongs on a circle of radius $\sqrt{R}$, and we may take $R$ as large as we like, it suffices to show that
\[
\lambda_H(2^{1/4}e^{3\pi i/8})>\lambda_H(2^{1/4}e^{\pi i/8}).
\]
By \eqref{lamb}, we see that this inequality holds and thus we have proved \eqref{omg1}. Since $K$ is compact, we deduce that there is some $\eta>0$ so that
\begin{equation}\label{omg2}
\rho_{\Omega_0^*}(z,-1)-\rho_{\Omega_0^*}(z,1)\geq \eta,
\end{equation}
for all $z\in K$.

Now let $a_0>R$ and note that $-a_0\notin\overline{D(-i,R)}$. Let $\gamma$ be the hyperbolic geodesic in $\Omega_0^*$ joining $-a_0$ to $-1$. By a theorem of J\o rgensen, see \cite{Jorg}, the open half planes $\{\Im z>-\epsilon\}$ are hyperbolically convex for any $\epsilon>0$ sufficiently small. Therefore, $\gamma$ is contained in $\{\Im z\geq 0\}$ and as such, it has to meet $K$ at least at one point $\zeta$. Then, by the triangle inequality, the fact that $\gamma$ is a geodesic, and \eqref{omg2},
\[
\rho_{\Omega_0^*}(-a_0,-1)-\rho_{\Omega_0^*}(-a_0,1)\geq \rho_{\Omega_0^*}(\zeta,-1)-\rho_{\Omega_0^*}(\zeta,1)\geq \eta.
\]
Let $\Omega_0=\C\setminus\{z:\ \Re z\leq a_0,\ \Im z =-1\}$. By conformal invariance, the last estimate can be written as
\[
\rho_{\Omega_0}(0,a_0-1)-\rho_{\Omega_0}(0,a_0+1)\geq \eta.
\]

We now construct the domain $\Omega_1$. For $a,b>0$ with $a-b>a_0+1$ and $-\frac{a}{b}\notin \overline{D(-i,R)}$, let
\[
\Omega_{a,b}=\Omega_0\setminus\{z:\ \Re z\leq a,\ \Im z=-b\}.
\]
As $a,b\to +\infty$, $\Omega_{a,b}$ converges in the sense of Carath\'eodory to $\Omega_0$. Thus, by Lemma \ref{PL2}, for $a,b$ large,
\begin{equation}\label{omg3}
\rho_{\Omega_{a,b}}(0,a_0-1)-\rho_{\Omega_{a,b}}(0,a_0+1)\geq \eta-\frac{\eta}{4}.
\end{equation}
The map $\frac{z-a}{b}$ maps $\Omega_{a,b}$ onto
\[
\Omega_{a,b}^*=\Omega_0^*\setminus\left\lbrace z:\ \Re z\leq \frac{a_0-a}{b},\ \Im z =\frac{-1}{b}\right\rbrace.
\]
Observe that we may require $\frac{a}{b}\to +\infty$, as $a,b\to +\infty$ so that $\Omega_{a,b}^*$ converges in the sense of Carath\'eodory to $\Omega_0^*$.
By conformal invariance,
\begin{equation}\label{omg4}
\rho_{\Omega_{a,b}}(0,a-b)-\rho_{\Omega_{a,b}}(0,a+b)=\rho_{\Omega_{a,b}^*}\left(-\frac{a}{b},-1\right)-\rho_{\Omega_{a,b}^*}\left(-\frac{a}{b},1\right).
\end{equation}
Let $\gamma$ be the hyperbolic geodesic in $\Omega_{a,b}^*$ joining $-\frac{a}{b}$ to $-1$. By J\o rgensen's theorem, $\gamma$ meets $K$ at some point $\zeta_1$, which depends on $a,b$. Then by the triangle inequality,
\begin{equation}\label{omg5}
\rho_{\Omega_{a,b}^*}\left(-\frac{a}{b},-1\right)-\rho_{\Omega_{a,b}^*}\left(-\frac{a}{b},1\right)\geq \rho_{\Omega_{a,b}^*}\left(\zeta_1,-1\right)-\rho_{\Omega_{a,b}^*}\left(\zeta_1,1\right).
\end{equation}
By Lemma \ref{PL2}, 
\[
\sup_{w\in K}\lvert \rho_{\Omega_{a,b}^*}\left(w,\pm 1\right)-\rho_{\Omega_0^*}(w,\pm 1)\rvert\to 0,
\]
as $a,b\to +\infty$. It follows that we can choose $a,b$ sufficiently large such that
\[
\rho_{\Omega_{a,b}^*}\left(\zeta_1,-1\right)-\rho_{\Omega_{a,b}^*}(\zeta_1,1)\geq 
\rho_{\Omega_0^*}(\zeta_1,-1)-\rho_{\Omega_0^*}(\zeta_1,1)-\frac{\eta}{4},
\]
which by \eqref{omg2} implies
\[
\rho_{\Omega_{a,b}^*}\left(\zeta_1,-1\right)-\rho_{\Omega_{a,b}^*}(\zeta_1,1)\geq \eta-\frac{\eta}{4}.
\]
Finally, \eqref{omg3}, and \eqref{omg4} together with \eqref{omg5}, show that there exist $a_1,b_1$ with $a_1-b_1>a_0+1$ so that if we set $\Omega_1=\Omega_{a_1,b_1}$, then
\[
\rho_{\Omega_1}(0,a_0-1)-\rho_{\Omega_1}(0,a_0+1)\geq \eta-\frac{\eta}{4}
\]
and 
\[
\rho_{\Omega_1}(0,a_1-b_1)-\rho_{\Omega_1}(0,a_1+b_1)\geq \eta-\frac{\eta}{4}.
\]

We repeat this process, by applying the argument above to $\Omega_k$ each time, in order to produce $\Omega_{k+1}$. This way, we obtain a sequence of convex in the positive direction simply connected domains $\Omega_n$ which satisfy, 
\[
\rho_{\Omega_n}(0,a_k-b_k)-\rho_{\Omega_n}(0,a_k+b_k)\geq \eta-\sum_{j=2}^{n+1}\frac{\eta}{2^j}>\frac{\eta}{2}.
\]
for all $0\leq k\leq n$. We set $\delta=\frac{\eta}{2}$ and the proof is complete.
\end{proof}

We can now prove the following theorem, which immediately implies Theorem \ref{Th2}.
\begin{theorem}\label{T2}
There exist a convex in the positive direction, simply connected domain $\tilde{\Omega}\supset\R$ and two sequences of real numbers $x_n,y_n$, $n\geq 0$,  satisfying $x_n<y_n<x_{n+1}$ and $x_n\to +\infty$, as $n\to +\infty$, so that
\[
\rho_{\tilde{\Omega}}(0,x_n)>\rho_{\tilde{\Omega}}(0,y_n),
\]
for all $n\geq 0$.
\end{theorem}
\begin{figure}
	\includegraphics[width=1\linewidth]{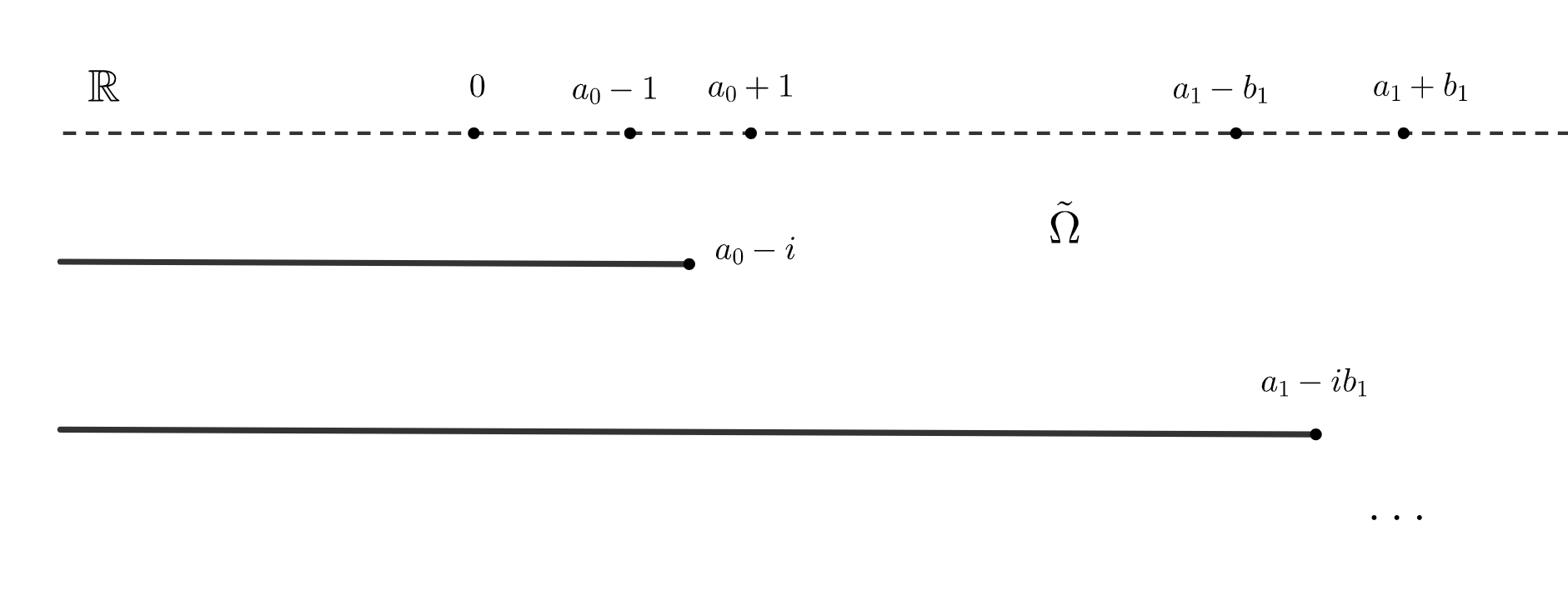}
	\caption{The domain $\tilde{\Omega}$ of Theorem \ref{T2}. The part shown coincides with $\Omega_1$.}
	\label{Fig2}	
\end{figure}
\begin{proof}
For $n\geq 0$, let $\Omega_n$, $a_n$, $b_n$, and $\delta$ be as in the statement of Lemma \ref{L2.1}. Set $\tilde{\Omega}=\bigcap_{j\geq 0}\Omega_j$ and take $x_n=a_n-b_n$, $y_n=a_n+b_n$, $n\geq 0$. Observe that $\tilde{\Omega}$ is simply connected, convex in the positive direction, and it contains the real line. See Figure \ref{Fig2}. Fix a non-negative integer $k$ and note that by Lemma \ref{L2.1},
\[
\rho_{\Omega_n}(0,x_k)-\rho_{\Omega_n}(0,y_k)\geq \delta,
\]
for all $n\geq k$. Moreover, $\tilde{\Omega}$ is the Carath\'eodory kernel of the sequence $\lbrace\Omega_n\rbrace_{n\geq 0}$. Therefore, letting $n\to +\infty$ in the last estimate and using Lemma \ref{PL2}, gives
\[
\rho_{\tilde{\Omega}}(0,x_k)-\rho_{\tilde{\Omega}}(0,y_k)\geq\delta>0.
\]
Since $k$ was arbitrary, the result follows.
\end{proof}

%%%%%%%%%%%%%%%%%%%%%%%%%%%%%%%%%%%%%%%%%%%%%%%%%%%%%%%%%%%%%%%%%%%%%%%%%%%%%%%%%%%%%%%%%
%%%%%%%%%%%%%%%%%%%%%%%%%%%%%%%%%%%%%%%%%%%%%%%%%%%%%%%%%%%%%%%%%%%%%%%%%%%%%%%%%%%%%%%%%
%%%%%%%%%%%%%%%%%%%%%%%%%%%%%%%%%%%%%%%%%%%%%%%%%%%%%%%%%%%%%%%%%%%%%%%%%%%%%%%%%%%%%%%%%
%%%%%%%%%%%%%%%%%%%%%%%%%%%%%%%%%%%%%%%%%%%%%%%%%%%%%%%%%%%%%%%%%%%%%%%%%%%%%%%%%%%%%%%%%
%%%%%%%%%%%%%%%%%%%%%%%%%%%%%%%%%%%%%%%%%%%%%%%%%%%%%%%%%%%%%%%%%%%%%%%%%%%%%%%%%%%%%%%%%

\section{The orthogonal speed is strictly increasing}
In this section we prove Theorem \ref{Th1}, which we restate here.
\begin{theorem}\label{T3}
Let $(\phi_t)$ be a non-elliptic semigroup in $\D$ with orthogonal speed $v^o$. The function $v^o$ is strictly increasing in $[0,+\infty)$.
\end{theorem}
\begin{proof}
Let $\tau$ be the Denjoy-Wolff point of $(\phi_t)$. Let $h:\D\to\C$ be the Koenigs function of the semigroup with $h(0)=0$ and let $\Omega=h(\D)$ be the Koenigs domain. Let $P_{\infty}=\hat{h}(\tau)$, where $\hat{h}$ is the Carath\'eodory extension of $h$, and let $\gamma=h((-\tau,\tau))$. Then $P_{\infty}$ is the prime end of $\Omega$ to which the slit $[0,\infty)$ converges, and $\gamma$ is the hyperbolic geodesic for $\Omega$ joining $P_{\infty}$ to some prime end $q\in\partial_C\Omega$ and containing the point $0$. By conformal invariance,
\[
v^o(t)=\rho_{\D}(0,\pi(\phi_t(0)))=\rho_{\Omega}(0,p(t)),\ t\geq 0,
\]
where $p(t)$ is the point of $\gamma$ with minimal hyperbolic distance from the point $t\geq 0$. If $\Omega$ is a horizontal half plane, then it is easy to check that $v^o$ is a strictly increasing function. Therefore, we assume for the rest of the proof that $\Omega$ is not a horizontal half plane.  It is known, see \cite[Theorem 3]{Bets}, that each vertical line intersects $\gamma$ once, at most.

\begin{figure}
	\includegraphics[width=1\linewidth]{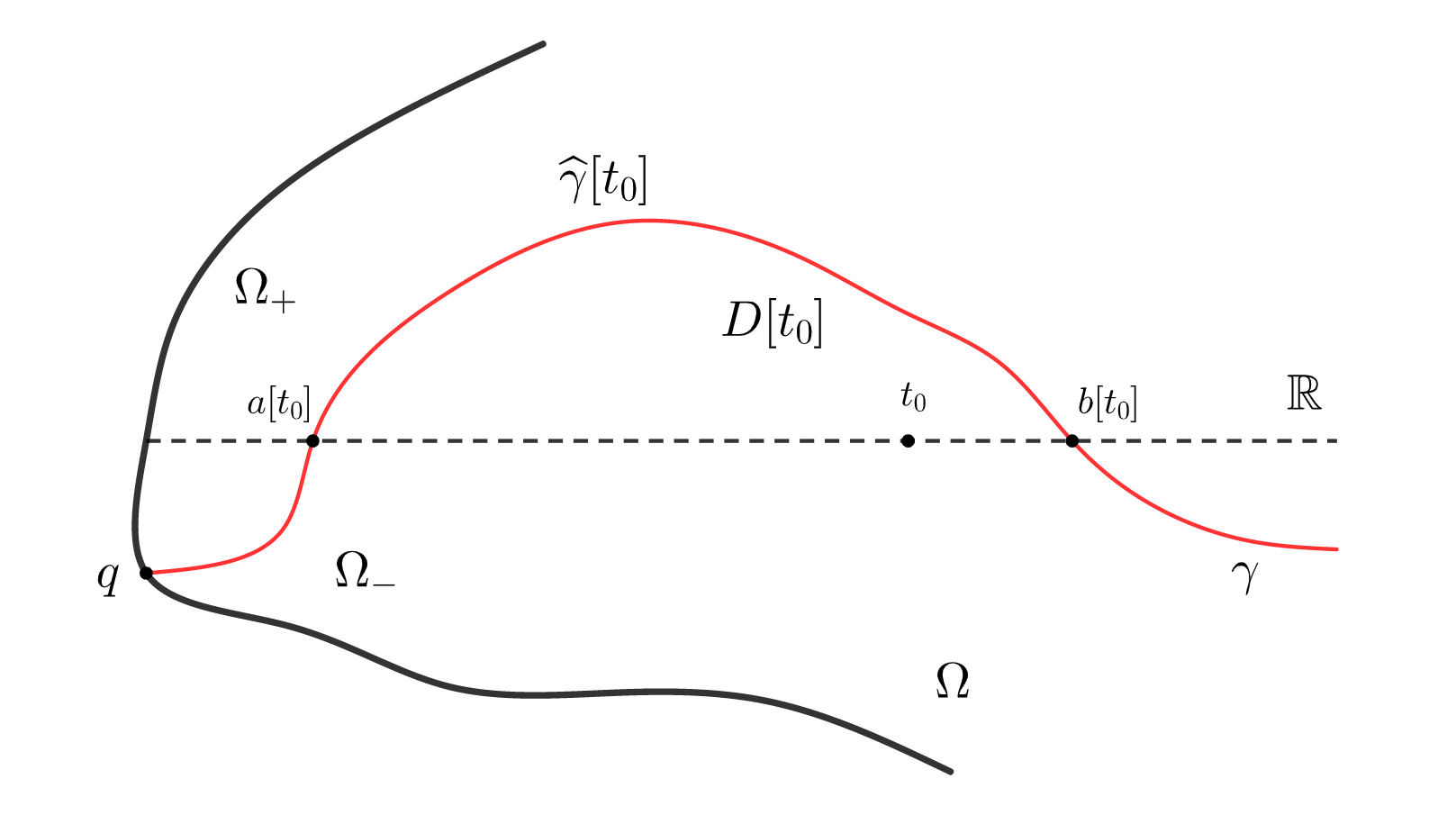}
	\caption{The red curve is $\gamma$. It separates $\Omega$ into $\Omega_{+}$ and $\Omega_{-}$. The part of the red curve joining $a[t_0]$ to $b[t_0]$ is $\widehat{\gamma}[t_0]$. Note that, in general, $q$ is a prime end of $\Omega$.}
	\label{Fig3}	
\end{figure}

If $\Omega$ is symmetric with respect to the real axis, then $\gamma$ contains the positive semiaxis and therefore,
for every $t\geq 0,$ $p(t)=t$. It follows that $v^o$ is strictly increasing. Thus,  we assume that  
$\Omega$ is not  symmetric with respect to  the real axis. Then the analytic curve $\gamma$ intersects $[0,\infty)$ at a finite or countably infinite number of points and the set $\gamma\cap [0,\infty)$ does not have a finite accumulation point. We now make some geometric considerations and set some notation. Let $\Omega_{-}$, $\Omega_{+}$ be the two subdomains of $\Omega$ determined by $\gamma$, i.e., the images, under $h$, of the two semidisks determined by the diameter $(-\tau,\tau)$ in $\mathbb{D}$. Suppose that $t_o\in (0,\infty)$ and $t_o\notin\gamma$. Let
$$
a[t_o]:=\max ([0,t_o)\cap\gamma)
$$
and
$$
b[t_o]:=\begin{cases} \min((t_o,\infty)\cap\gamma),\;\;\;\;&\hbox{if}\;\;\;\;(t_o,\infty)\cap\gamma\neq\varnothing,
\\
P_\infty, &\hbox{if} \;\;\;\;(t_o,\infty)\cap\gamma=\varnothing.
\end{cases}
$$
Consider the geodesic segment (subset of $\gamma$), $\widehat{\gamma}[t_o]:=(a[t_o],b[t_o])_\Omega$
and the simply connected domain $D[t_o]$ with
$$
\partial D[t_o]=\widehat{\gamma}[t_o]\cup [a[t_o],b[t_o]].
$$
Note that, since  $\Omega$ is  convex in the positive direction, $D[t_o]\subset \Omega$. Moreover, either $D[t_o]\subset \Omega_-$ or $D[t_o]\subset \Omega_+$. See Figure \ref{Fig3}. If there is no danger of confusion, we will use the simpler pieces of notation $a,b, \widehat{\gamma}, D$ instead of   $a[t_o],b[t_o], \widehat{\gamma}[t_o], D[t_o]$. Sometimes, in addition to the hyperbolic segment $(a,b)_\Omega$, we will use the Euclidean segment (interval) $(a,b)$. In this case, if $b=P_\infty$ then by $(a,b)$ we denote the interval $(a,\infty)$. 

 Since each vertical line intersects $\gamma$ once, at most, we infer that $v^o$ is strictly increasing if and only if the function $\Re p(t)$ is strictly increasing in $[0,+\infty)$. 
 For the sake of contradiction, we assume that there exist $t_1,t_2$ such that
\begin{equation}\label{Assumption}
0\leq t_1<t_2 \;\;\;\hbox{and}\;\;\;
\Re p(t_1)\geq \Re p(t_2). 
\end{equation}
We now consider the following cases.

\medskip

{\bf Case 1:} $t_1\notin \gamma$ and $p(t_1)\notin \widehat{\gamma}[t_1]$. See Figure \ref{Fig4}.\\
We are going to show that Case 1 cannot occur. Since $p(t_1)\notin \widehat{\gamma}[t_1]$, we have that either $p(t_1)\in (q,a[t_1])_\Omega$ or
$p(t_1)\in (b[t_1],P_\infty)_\Omega$. We assume that $p(t_1)\in (b[t_1],P_\infty)_\Omega$; the other case is similar. We further assume that $t_1\in\Omega_-$; the case $t_1\in \Omega_+$ is treated in a similar manner.

%\begin{figure}
%	\includegraphics[width=1\linewidth]{Fig4.png}
%	\caption{The situation in Case 1; $t_1\in\Omega_{-}$ and $p(t_1)\in \left(b[t_1], P_{\infty}\right)_{\Omega}$. The dotted red arc is the geodesic segment in $\Omega$ joining $t_1$ to $p(t_1)$ and it is orthogonal to $\gamma$.}
%	\label{Fig4}	
%\end{figure}
\medskip

We will use the notation $a,b, \widehat{\gamma}, D$ for  $a[t_1],b[t_1], \widehat{\gamma}[t_1], D[t_1]$.
Consider the domain $\widetilde{\Omega}$ which is the component of $(D\cup (a,b)\cup D^*)\cap\Omega$ 
containing $t_1$. Recall that the superscript $*$ denotes reflection in the real line. Note that 
$\widetilde{\Omega}$ has the following properties:\\
(i) It is simply connected,\\
(ii) $\widetilde{\Omega}\subset\Omega_-$,\\
(iii) $(\widetilde{\Omega}\cap \{z:\Im z<0\})^*=(D^*\cap \Omega_-)^*\subset D=\widetilde{\Omega}\cap \{z:\Im z>0\}$.\\
Note also that the hyperbolic segments $[p(t_1),P_\infty)_\Omega$ and  $\widehat{\gamma}$ are disjoint.

 By the Strong Markov Property,
\begin{eqnarray}\label{m3p1}
\omega(t_1,[p(t_1),P_\infty)_\Omega,\Omega_-)&=&\int_{\widehat{\gamma}^*\cap\partial \widetilde{\Omega}} \omega(t_1,d\zeta,\widetilde{\Omega})\; \omega(\zeta,[p(t_1),P_\infty)_\Omega,\Omega_-) \nonumber \\
&<& \omega(t_1,\widehat{\gamma}^*\cap\partial \widetilde{\Omega},\widetilde{\Omega}).
\end{eqnarray}
The inequality is strict because for every $\zeta\in \widehat{\gamma}^*\cap\partial \widetilde{\Omega}$, 
$$
\omega(\zeta,[p(t_1),P_\infty)_\Omega,\Omega_-)<1.
$$ 
Also, because of  the property (iii), we may apply Lemma \ref{RL}(b) to conclude that for $z\in (a,b)\cup(D^*\cap \widetilde{\Omega})$,
\begin{equation}\label{m3p2}
\omega(z,\widehat{\gamma}^*\cap\partial \widetilde{\Omega}, \widetilde{\Omega})\leq 
\omega(\bar{z},\widehat{\gamma}, \widetilde{\Omega}).
\end{equation}
In particular,
\begin{equation}\label{m3p3}
\omega(t_1,\widehat{\gamma}^*\cap\partial \widetilde{\Omega}, \widetilde{\Omega})\leq 
\omega(t_1,\widehat{\gamma}, \widetilde{\Omega}).
\end{equation}
By the domain monotonicity of harmonic measure,
\begin{eqnarray}\label{m3p4}
\omega(t_1,\widehat{\gamma}, \widetilde{\Omega})&\leq & \omega(t_1,\widehat{\gamma}, \Omega_-) 
< \omega(t_1,(q,p(t_1)]_\Omega, \Omega_-)\\ &=&\omega(t_1,[p(t_1),P_\infty)_\Omega, \Omega_-),\nonumber
\end{eqnarray}
where the last equality follows from Lemma \ref{PL3}. 
Combining (\ref{m3p1}), (\ref{m3p3}), and (\ref{m3p4}), we arrive at a contradiction. Thus, as we claimed, Case 1 is impossible. 
\begin{figure}
	\includegraphics[width=1\linewidth]{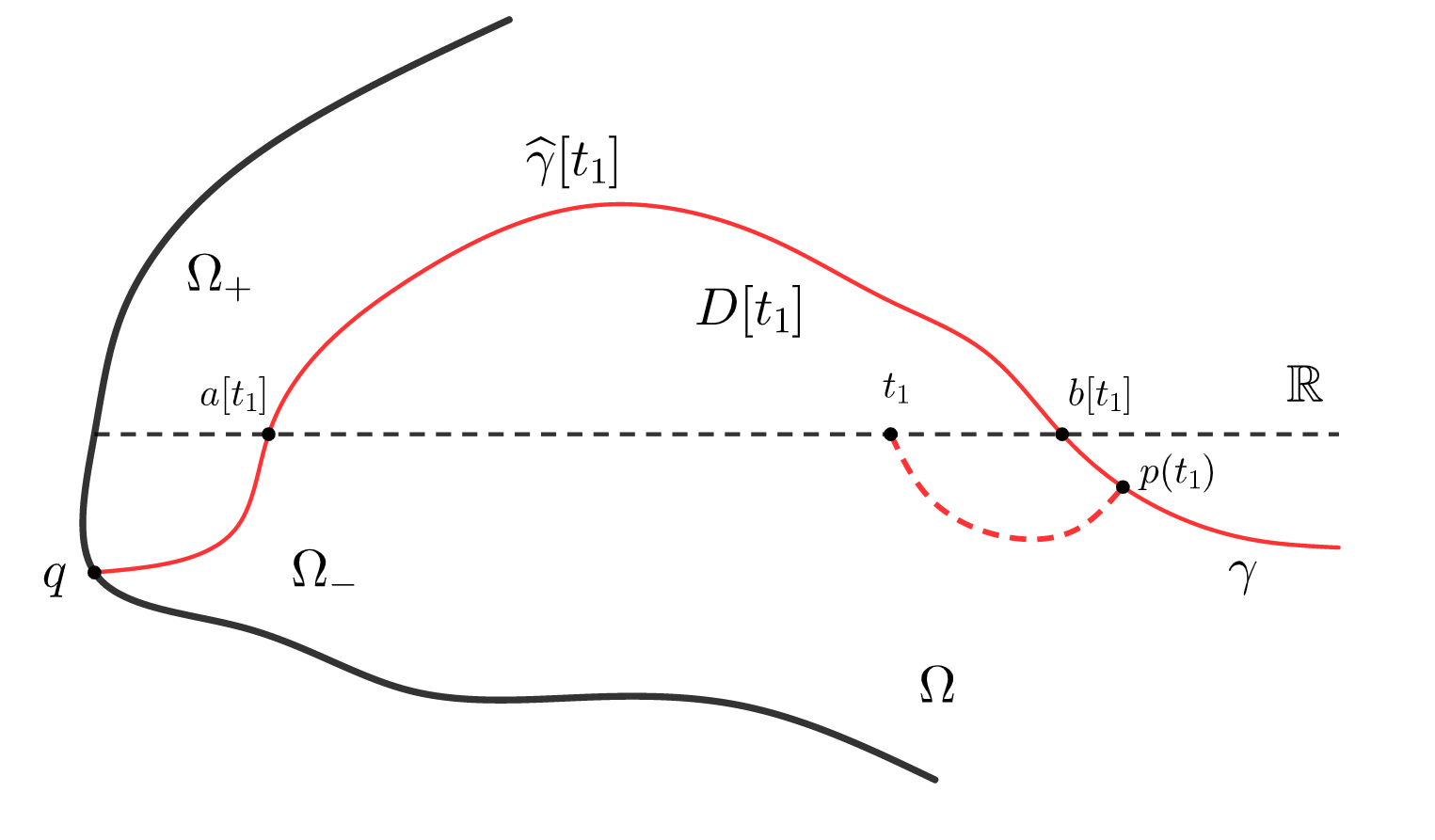}
	\caption{The situation in Case 1; $t_1\in\Omega_{-}$ and $p(t_1)\in \left(b[t_1], P_{\infty}\right)_{\Omega}$. The dotted red arc is the geodesic segment in $\Omega$ joining $t_1$ to $p(t_1)$ and it is orthogonal to $\gamma$.}
	\label{Fig4}	
\end{figure}
\medskip

{\bf Case 2:} $t_1\notin \gamma$ and  $p(t_1)\in \widehat{\gamma}[t_1]$. \\
We further assume that $t_1\in \Omega_-$; the case $t_1\in\Omega_+$ is similar.

\medskip

{\bf Subcase 2.1:} $t_2\in\gamma$.\\
Then $t_2=p(t_2)$. Since $p(t_1)\in  \widehat{\gamma}[t_1]$, we conclude that $\Re p(t_1)<\Re p(t_2)$ which contradicts (\ref{Assumption}). 

\medskip

{\bf Subcase 2.2:} $t_2\notin\gamma$ and $p(t_2)\notin \widehat{\gamma}[t_2]$.\\
We arrive at a contradiction using the same argument as in Case 1 (with $t_2$ playing now the role of $t_1$). 

\medskip

{\bf Subcase 2.3:} $t_2\notin\gamma$ and $p(t_2)\in \widehat{\gamma}[t_2]$.\\
We claim that $\widehat{\gamma}[t_1]= \widehat{\gamma}[t_2]$. Indeed, suppose that 
 $\widehat{\gamma}[t_1]\neq\widehat{\gamma}[t_2]$. Then  $\widehat{\gamma}[t_1]\cap\widehat{\gamma}[t_2]=\varnothing$ and hence
$$
\Re p(t_1)<b[t_1]\leq a[t_2]<\Re p(t_2) 
$$
which contradicts (\ref{Assumption}). So we have $\widehat{\gamma}[t_1]= \widehat{\gamma}[t_2]$. We set
$a=a[t_1]=a[t_2]$, $b=b[t_1]=b[t_2]$, $\widehat{\gamma}=\widehat{\gamma}[t_1]=\widehat{\gamma}[t_2]$, $D=D[t_1]=D[t_2]$. 

Consider the geodesic curves (for $\Omega$) $\gamma_1,\gamma_2$ that are orthogonal to $\gamma$ and $t_1\in\gamma_1$, $t_2\in \gamma_2$. Note that either $\gamma_1\cap\gamma_2=\varnothing$ or $\gamma_1=\gamma_2$. We first deal with the former possibility:
\begin{figure}
	\includegraphics[width=1\linewidth]{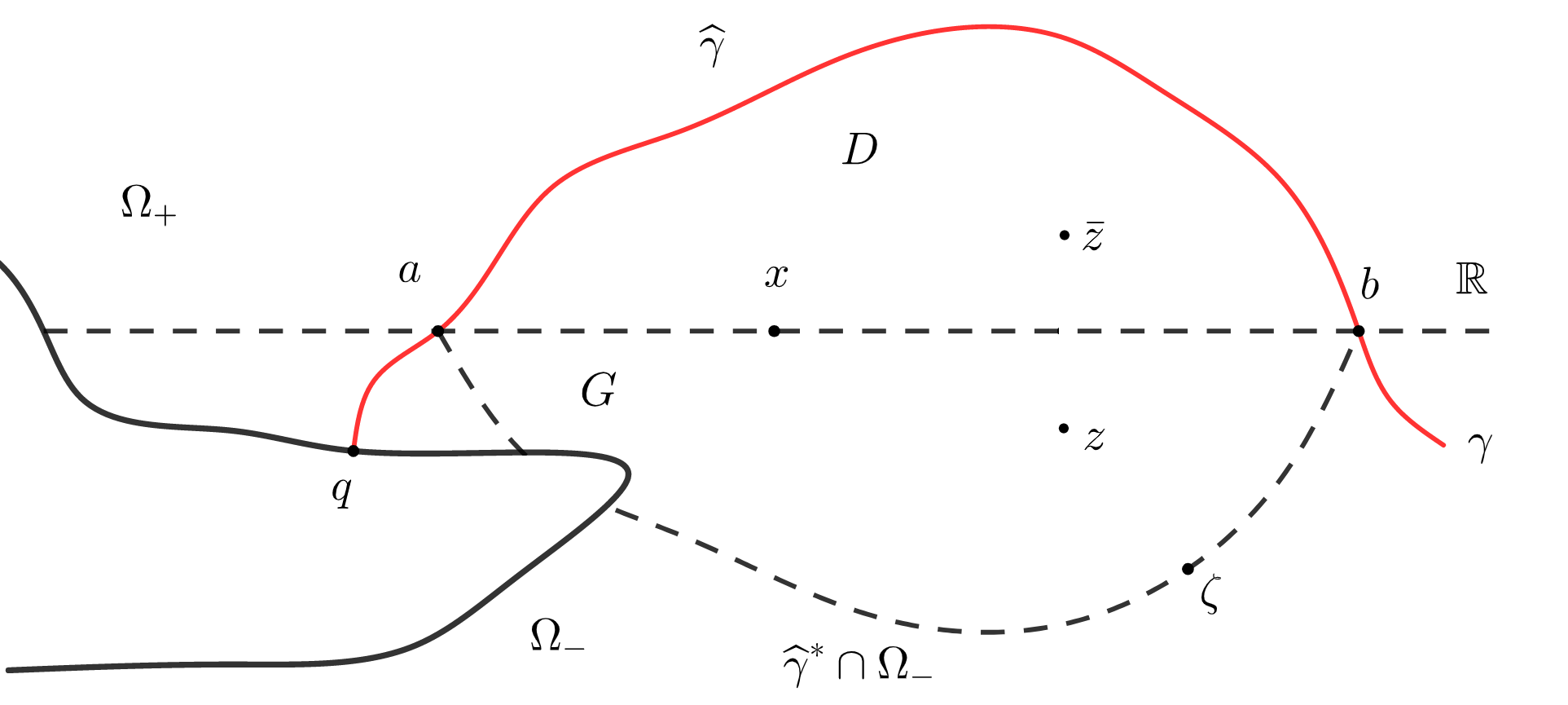}
	\caption{The union of the two dotted arcs in the lower half plane is the set $\widehat{\gamma}^*\cap\Omega_{-}$.}
	\label{Fig5}	
\end{figure}
\medskip

{\bf Subsubcase 2.3.A:} $\gamma_1\cap\gamma_2=\varnothing$.

\noindent
 Consider the function
\begin{equation}\label{m3p10}
\omega(z):=\omega(z,\gamma,\Omega_-),\;\;\;z=x+iy\in\Omega_-.
\end{equation}
Let $G$ be the component of $D^*\cap\Omega_{-}$ having $t_1$ on its boundary. See Figure \ref{Fig5}. Note that since $\Omega$ is convex in the positive direction, the domain $G$ is simply connected.
By the strong Markov property, for every $z\in G$,
\begin{equation}\label{m3p13}
\omega(\bar{z})=\omega(\bar{z}, \widehat{\gamma},D)+\int_{(a,b)}\omega(\bar{z},dx,D)\;\omega(x)
\end{equation}
and
\begin{equation}\label{m3p14}
\omega(z)=\int_{\widehat{\gamma}^*\cap \partial G}\omega(z, d\zeta,G)\,\omega(\zeta)  +\int_{(a,b)}\omega(z,dx,G)\;\omega(x).
\end{equation}
\begin{figure}
	\includegraphics[width=1\linewidth]{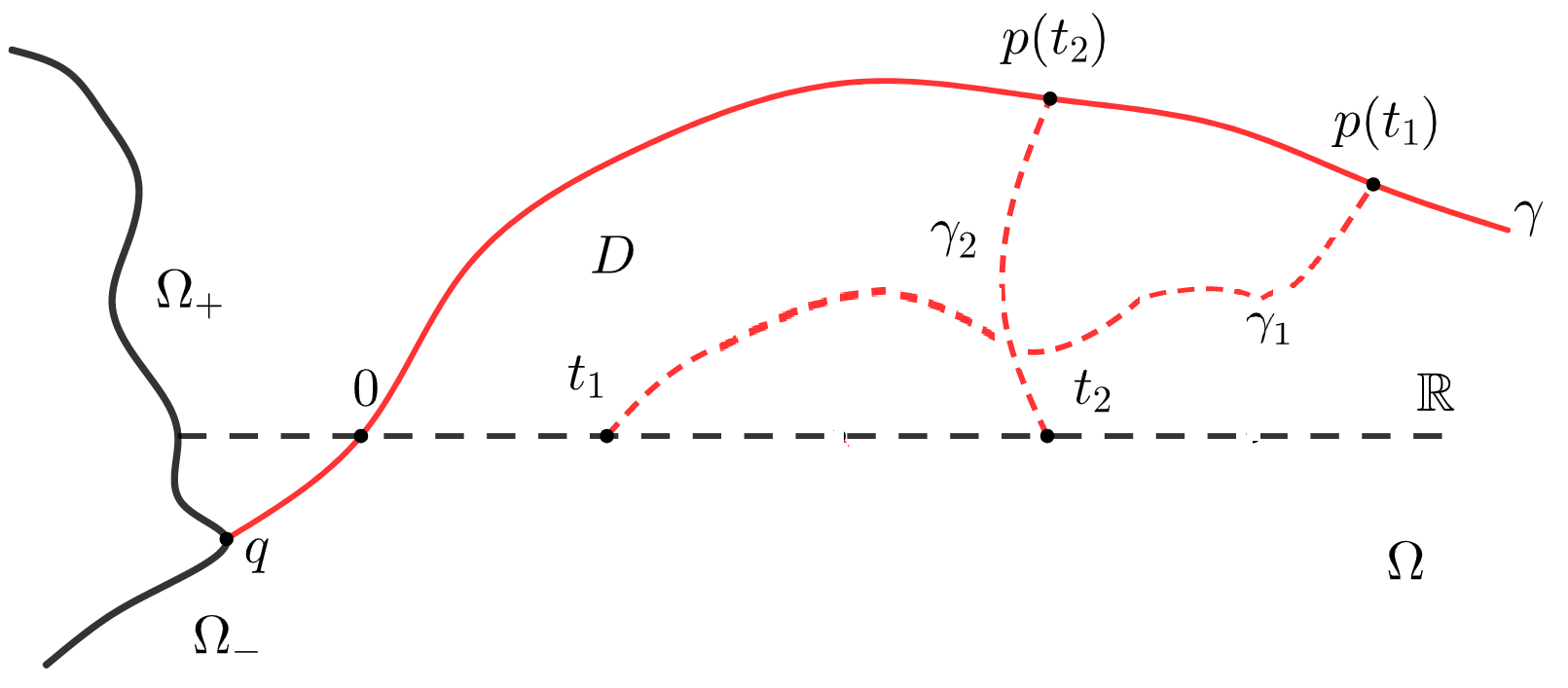}
	\caption{The situation in Subsubcase 2.3.A; $t_1\in\Omega_{-}$. $t_1,\ t_2\notin\gamma$, $p(t_1),\ p(t_2)\in \widehat{\gamma}[t_1]=\widehat{\gamma}[t_2]$.}
	\label{Fig6}	
\end{figure}

 Since $G^*\subset D$, we may apply Lemma \ref{RL}(a) to conclude that for every Borel set $B\subset (a,b)$ and every $z\in G$,
\begin{equation}\label{m3p15}
\omega(z, B,G)\leq \omega(\bar{z}, B,D).
\end{equation}
Also, by Lemma \ref{RL}(c), for every $z\in G$,
\begin{equation}\label{m3p16}
\int_{\widehat{\gamma}^*\cap \partial G}\omega(z, d\zeta,G)\,\omega(\zeta)< 
\omega(z, \widehat{\gamma}^*\cap \partial G, G) 
\leq  \omega(\bar{z}, \widehat{\gamma}, D).
\end{equation}
The first inequality in (\ref{m3p16}) is strict because $\omega(\zeta)<1$ for every $\zeta\in \widehat{\gamma}^*\cap \partial G$.
It follows from (\ref{m3p13})-(\ref{m3p16}) that for every $z\in G$, we have
\begin{equation}\label{m3p17}
\omega(z)<\omega(\bar{z}).
\end{equation}

Consider the function $u(z)=\omega(z)-\omega(\bar{z}), \;\;z\in G$. This is a harmonic function in $G$ and, by (\ref{m3p17}), $u<0$ in $G$. Moreover, $u$ extends continuously on $(a,b)$ and $u(t)=0$ for every $t\in (a,b)$. By Hopf's lemma (see e.g. \cite[Lemma 3.4]{GT}),
$u_y(t)>0$ for every $t\in (a,b)$. This means that 
\begin{equation}\label{m3p20}
\omega_y(t)>0,\;\;\;\;t\in (a,b).
\end{equation}

\medskip

Let $\epsilon(\zeta)=(\epsilon_1(\zeta),\epsilon_2(\zeta))$ be the unit tangent vector of $\gamma_1$ at $\zeta\in\gamma_1$; ($\gamma_1$ is assumed to have the orientation from $t_1$ to $p(t_1)$). By Lemma \ref{PL4},
\begin{equation}\label{m3p11}
\nabla \omega(\zeta)=(\omega_x(\zeta),\omega_y(\zeta))=\lambda\;\epsilon(\zeta),\;\;\;\lambda>0,\;\zeta\in [t_1,p(t_1))_\Omega.
\end{equation}
By  (\ref{m3p20}), $\epsilon_2(t)>0$, for every $t\in (a,b)\cap [t_1,p(t_1)]_\Omega$. 
It follows that $(t_1, p(t_1))_\Omega\subset D$ and similarly $(t_2, p(t_2))_\Omega\subset D$. 
Therefore,  $[t_1, p(t_1)]_\Omega$ and  $[t_2, p(t_2)]_\Omega$ are crosscuts of $D$. Note that $D$ is a Jordan domain and that the points $t_1, t_2, p(t_1), p(t_2)$ lie on the Jordan curve $\partial D$ in this order. See Figure \ref{Fig6}. The crosscut  $[t_1, p(t_1)]_\Omega$ divides $D$ into two disjoint Jordan domains $D_1$, $D_2$ such that $\partial D_1\cap\partial D_2=[t_1,p(t_1)]_\Omega$, $p(t_2)\in \partial D_1\setminus [t_1,p(t_1)]_\Omega$ and $t_2\in \partial D_2\setminus [t_1,p(t_1)]_\Omega$.
Thus the crosscut $[t_2,p(t_2)]_\Omega$ intersects $[t_1,p(t_1)]_\Omega$. This cannot happen because $\gamma_1\cap\gamma_2=\varnothing$.

%\begin{figure}
%	\includegraphics[width=1\linewidth]{Fig6.png}
%	\caption{The union of the two dotted arcs in the lower half plane is the set $\widehat{\gamma}^*\cap\Omega_{-}$.}
%	\label{Fig6}	
%\end{figure}

\medskip

{\bf Subsubcase 2.3.B:} $\gamma_1=\gamma_2$. See Figure \ref{Fig7}.

\begin{figure}
	\includegraphics[width=1\linewidth]{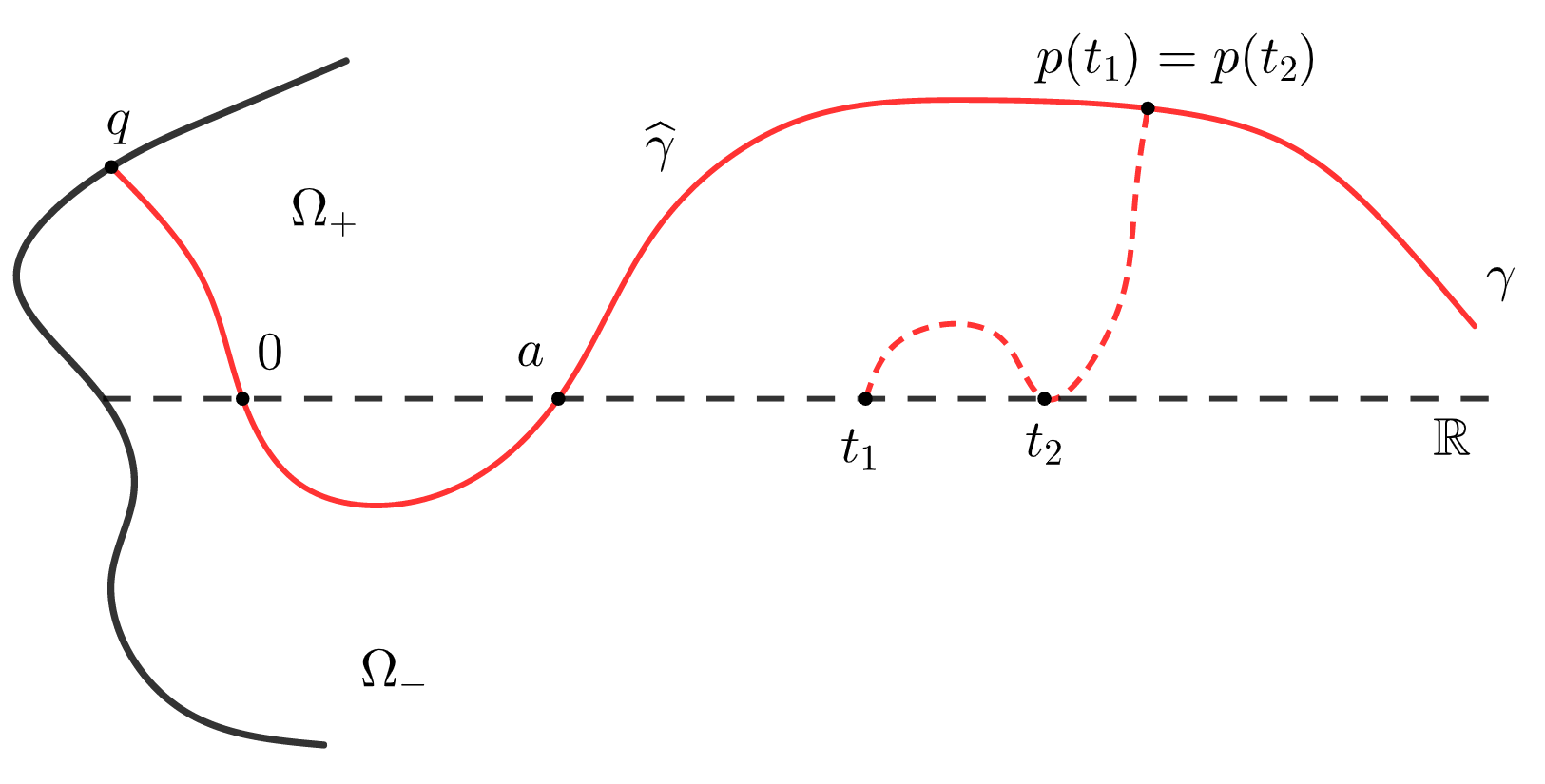}
	\caption{The situation in subsubcase 2.3.B; $\gamma_1=\gamma_2$ and $p(t_1)=p(t_2)\in \widehat{\gamma}$.}
	\label{Fig7}	
\end{figure}

\noindent
In this situation, we have $p(t_1)=p(t_2)$ and either $[t_1,p(t_1)]_\Omega\subset [t_2,p(t_2)]_\Omega$ or 
$[t_2,p(t_2)]_\Omega\subset [t_1,p(t_1)]_\Omega$. We assume that $[t_2,p(t_2)]_\Omega\subset [t_1,p(t_1)]_\Omega$; the other possibility is treated similarly. As we saw in Subsubcase 2.3.A.,  $[t_1,p(t_1)]_\Omega$ lies in  $\{z:\Im z\geq 0\}$.

Let $\epsilon(\zeta)=(\epsilon_1(\zeta),\epsilon_2(\zeta))$ be the unit tangent vector of $\gamma_1$ at $\zeta\in\gamma_1$; ($\gamma_1$ is assumed to have the orientation from $t_1$ to $p(t_1)$). Since the analytic arc $[t_1,p(t_1)]_\Omega$ lies in  $\{z:\Im z\geq 0\}$ and contains the point $t_2\in \R$, we have $\epsilon_2(t_2)=0$. As we saw in Subsubcase 2.3.A (with the argument with the harmonic measure and Hopf's lemma), this cannot happen.

\medskip

{\bf Case 3:} $t_1\in\gamma$.\\
Then $t_1=p(t_1)$. 

\medskip

{\bf Subcase 3.1:} $t_2\in\gamma$.\\
Then $t_2=p(t_2)$. Therefore, $\Re p(t_1)<\Re p(t_2)$ which contradicts (\ref{Assumption}).

\medskip

{\bf Subcase 3.2:} $t_2\notin\gamma$.\\
If $p(t_2)\notin \widehat{\gamma}[t_2]$, then this subcase coincides with Subcase 2.2. 
If $p(t_2)\in \widehat{\gamma}[t_2]$, then 
$$
\Re p(t_1)=p(t_1)=t_1\leq a[t_2]<\Re p(t_2)
$$
which contradicts  (\ref{Assumption}).
\end{proof}

\section{A domain monotonicity property of the orthogonal speed}

In this section we will prove Theorem \ref{Th4} which we restate.
\begin{theorem}
Suppose that $(\phi_t)$, $(\widetilde{\phi}_t)$ are 
semigroups with Denjoy-Wolff points $\tau,\widetilde{\tau}\in\partial \mathbb D$, Koenigs domains $\Omega,\widetilde{\Omega}$, and orthogonal speeds 
$v^o, \widetilde{v}^o$, respectively. If $\Omega\subset\widetilde{\Omega}$, then
\begin{equation}\label{ls1}
\liminf_{t\to\infty}[v^o(t)-\widetilde{v}^o(t)]\geq -\log 2.
\end{equation} 
\end{theorem}

Before we proceed with the proof, we need the following observation. For convenience, we will simplify the notation $\pi(\phi_t(0))$ to $\pi_t$. Note that
\[
v^o(t)=\rho_{\D}(0,\pi_t)=\frac{1}{2}\log\frac{1+\pi_t}{1-\pi_t}.
\]
By Theorem \ref{Th1}, $v^o$ is a strictly increasing function of $t>0$ and therefore so is $\pi_t$. Let $\Gamma_t$ be the hyperbolic geodesic in $\D$ which is perpendicular to the radial segment $(0,1)$ and passes through the point $\phi_t(0)$. Let $R_t$ denote the component of $\D\setminus \Gamma_t$ with the point $1$ on its boundary. The fact that $\pi_t$ is strictly increasing implies that the arc $\{\phi_s(0),\; s>t\}$, is contained in $R_t$. 

\begin{figure}
	\includegraphics[width=0.5\linewidth]{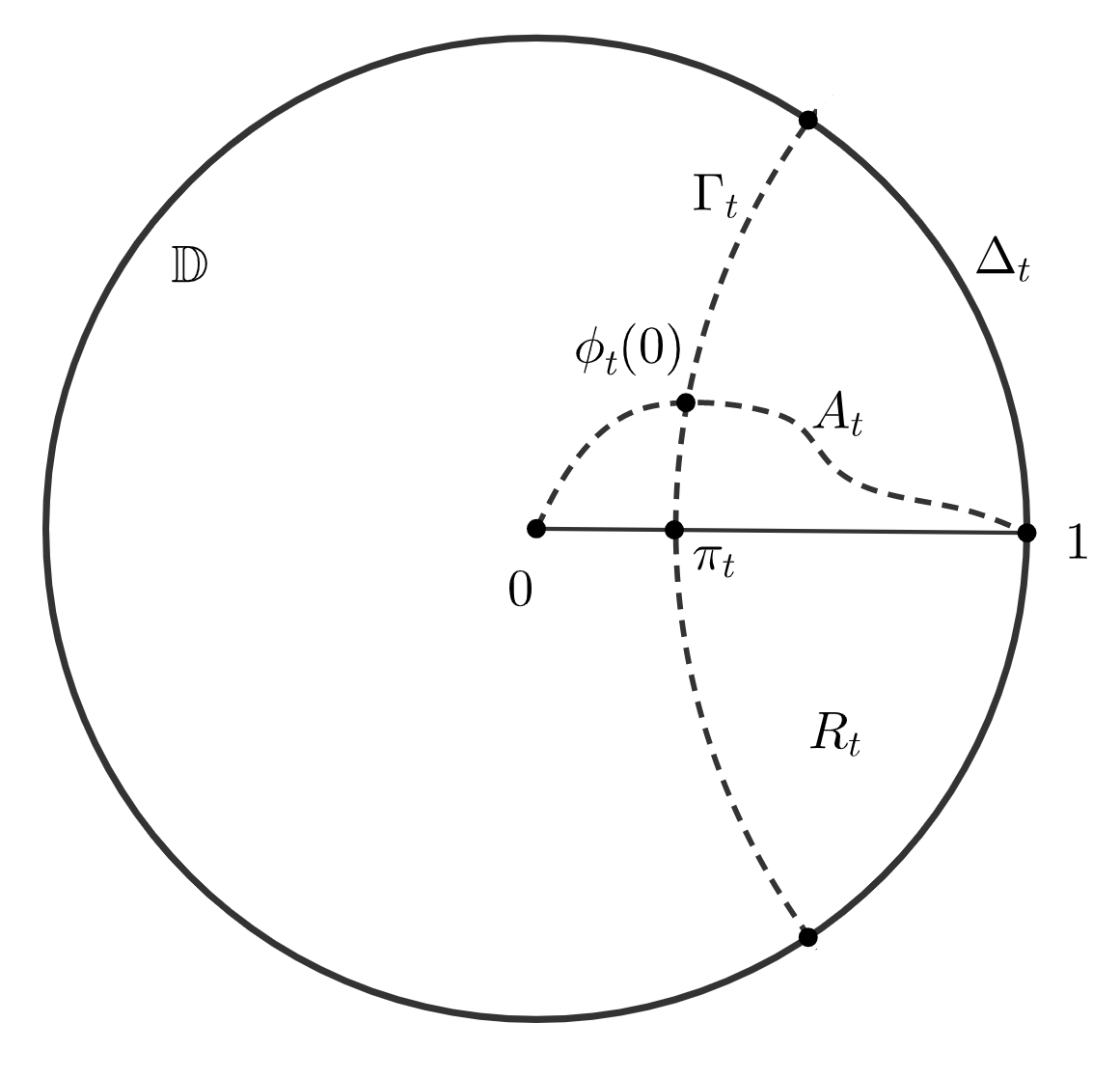}
	\caption{The circular arc $\Gamma_t$ is perpendicular to both $\partial\D$ and $(0,1)$. The arc $A_t$ is contained in $R_t$.}
	\label{Fig8}	
\end{figure}

\begin{proof}
By working with the semigroups defined by $\psi_t(z)=\overline{\tau}\phi_t(\tau z)$ and $\widetilde{\psi}_t(z)=\overline{\widetilde{\tau}}\widetilde{\phi}_t(\widetilde{\tau} z)$, if necessary, we may assume that $\tau=\widetilde{\tau}=1$. Let $h:\D\to\Omega$, $\widetilde{h}:\D\to\widetilde{\Omega}$ be the Koenigs maps for the semigroups. For $t>0$, let
\[
A_t=\{z\in\D:\ h(z)\in [t,+\infty)\},
\]
and
\[
\widetilde{A}_t=\{z\in\D:\ \widetilde{h}(z)\in [t,+\infty)\}.
\]
Note that since $t>0$, we have $0\notin A_t$ and $0\notin \widetilde{A}_t$. 
By the domain monotonicity and conformal invariance of harmonic measure,
\begin{equation}\label{ls2}
\omega\left(0,A_t,\D\setminus A_t\right)\leq \omega\left(0,\widetilde{A}_t,\D\setminus\widetilde{A}_t\right).
\end{equation}
Let $\widetilde{\Gamma}_t$ be the hyperbolic geodesic in $\D$ which is perpendicular to $(0,1)$ (at $\widetilde{\pi}_t$) and passes through the point $\widetilde{\phi}_t(0)$. Let $\widetilde{R}_t$ be the component of $\D\setminus\widetilde{\Gamma}_t$ with $1$ on its boundary, and let $\widetilde{\Delta}_t=\partial\widetilde{R}_t\setminus\widetilde{\Gamma}_t$. We also denote by $\Gamma_t,\ R_t,\ \Delta_t$ the corresponding sets for $(\phi_t)$. See Figure \ref{Fig8}. Note that $\Delta_t,\ \widetilde{\Delta}_t\subset\partial\D$. By the observation above, $\widetilde{A}_t\subset  \overline{\widetilde{R}_t}$ and thus,
by the extended maximum principle (see e.g. \cite[Theorem 3.6.9.]{Ran}),
\[
\omega\left(0,\widetilde{A}_t,\D\setminus\widetilde{A}_t\right)\leq \omega\left(0,\widetilde{\Gamma}_t,\D\setminus \overline{\widetilde{R}_t}\right).
\]
It is not hard to see, for example by transferring the situation to the upper half plane, that
\[
\omega\left(0,\widetilde{\Gamma}_t,\D\setminus \overline{\widetilde{R}_t}\right)=2\omega\left(0,\widetilde{\Delta}_t,\D\right).
\]
By a direct calculation,
\[
\omega\left(0,\widetilde{\Delta}_t,\D\right)=\frac{1}{\pi}\tan^{-1}\left(\frac{1-\widetilde{\pi}_t^2}{2\widetilde{\pi}_t}\right).
\]
We conclude that
\begin{equation}\label{ls3}
\omega\left(0,\widetilde{A}_t,\D\setminus\widetilde{A}_t\right)\leq \frac{2}{\pi}\tan^{-1}\left(\frac{1-\widetilde{\pi}_t^2}{2\widetilde{\pi}_t}\right).
\end{equation}

We now estimate $\omega\left(0,A_t,\D\setminus A_t\right)$ from below. By a projection theorem for the harmonic measure (see \cite[Theorem 7.2.13]{BCD} and references therein), and the observation preceding the proof,
\[
\omega\left(0,A_t,\D\setminus A_t\right)\geq \omega\left(0, \Delta_t\cap\Ha_{-},\D\right).
\]
Therefore,
\begin{equation}\label{ls4}
\omega\left(0,A_t,\D\setminus A_t\right)\geq \frac{1}{2\pi}\tan^{-1}\left(\frac{1-\pi_t^2}{2\pi_t}\right).
\end{equation}
By \eqref{ls2}, \eqref{ls3}, and \eqref{ls4},
\begin{equation}\label{ls5}
\tan^{-1}\left(\frac{1-\pi_t^2}{2\pi_t}\right)\leq 4\tan^{-1}\left(\frac{1-\widetilde{\pi}_t^2}{2\widetilde{\pi}_t}\right).
\end{equation}
Set $\frac{1-\pi_t^2}{2\pi_t}=x_t$ and $\frac{1-\widetilde{\pi}_t^2}{2\widetilde{\pi}_t}=\widetilde{x}_t$. Then we may write \eqref{ls5} as
\[
\frac{\tan^{-1}x_t}{x_t}\leq 4\frac{\widetilde{x}_t}{x_t}\frac{\tan^{-1}\widetilde{x}_t}{\widetilde{x}_t}.
\]
Observe that $\pi_t,\ \widetilde{\pi}_t\to 1$ and thus $x_t,\ \widetilde{x}_t\to 0$, as $t\to +\infty$. It follows that
\[
\liminf_{t\to\infty}\frac{\widetilde{x}_t}{x_t}\geq \frac{1}{4},
\]
which is equivalent to
\begin{equation}\label{ls6}
\liminf_{t\to\infty}\frac{1-\widetilde{\pi}_t^2}{1-\pi_t^2}\geq\frac{1}{4}.
\end{equation}
Finally, in view of the identity
\[
v^o(t)-\widetilde{v}^o(t)=\frac{1}{2}\log\left[ \frac{1-\widetilde{\pi}_t^2}{1-\pi_t^2}\left(\frac{1+\pi_t}{1+\widetilde{\pi}_t}\right)^2\right],
\]
\eqref{ls6} implies \eqref{ls1}.
\end{proof}

%%%%%%%%%%%%%%%%%%%%%%%%%%%%%%%%%%%%%%%%%%%%%%%%%%%%%%%%%%%%%%%%%%%%%%%%%%%%%%%%%%%%%%%%%
%%%%%%%%%%%%%%%%%%%%%%%%%%%%%%%%%%%%%%%%%%%%%%%%%%%%%%%%%%%%%%%%%%%%%%%%%%%%%%%%%%%%%%%%%
%%%%%%%%%%%%%%%%%%%%%%%%%%%%%%%%%%%%%%%%%%%%%%%%%%%%%%%%%%%%%%%%%%%%%%%%%%%%%%%%%%%%%%%%%
%%%%%%%%%%%%%%%%%%%%%%%%%%%%%%%%%%%%%%%%%%%%%%%%%%%%%%%%%%%%%%%%%%%%%%%%%%%%%%%%%%%%%%%%%
%%%%%%%%%%%%%%%%%%%%%%%%%%%%%%%%%%%%%%%%%%%%%%%%%%%%%%%%%%%%%%%%%%%%%%%%%%%%%%%%%%%%%%%%%

\bigskip

{\bf Acknowledgements:} We thank Dmitry Yakubovich, Davide Cordella, and the referees for their remarks and corrections.

%%%%%%%%%%%%%%%%%%%%%%%%%%%%%%%%%%%%%%%%%%%%%%%%%%%%%%%%%%%%%%%%%%%%%%%%%%%%%%%%%%%%%%%%%
%%%%%%%%%%%%%%%%%%%%%%%%%%%%%%%%%%%%%%%%%%%%%%%%%%%%%%%%%%%%%%%%%%%%%%%%%%%%%%%%%%%%%%%%%
%%%%%%%%%%%%%%%%%%%%%%%%%%%%%%%%%%%%%%%%%%%%%%%%%%%%%%%%%%%%%%%%%%%%%%%%%%%%%%%%%%%%%%%%%
%%%%%%%%%%%%%%%%%%%%%%%%%%%%%%%%%%%%%%%%%%%%%%%%%%%%%%%%%%%%%%%%%%%%%%%%%%%%%%%%%%%%%%%%%
%%%%%%%%%%%%%%%%%%%%%%%%%%%%%%%%%%%%%%%%%%%%%%%%%%%%%%%%%%%%%%%%%%%%%%%%%%%%%%%%%%%%%%%%%

\end{document}